\documentclass[
final
]{dmtcs-episciences}

\usepackage{tikz}
\usetikzlibrary{decorations.pathmorphing, decorations.pathreplacing}
\usetikzlibrary{calc, arrows}
\usetikzlibrary{intersections}

\usepackage{amsmath}

\usepackage[utf8]{inputenc}
\usepackage{subfigure}

\def\oV{{\widetilde V}}

\newtheorem{theorem}{Theorem}[section]     

\newtheorem{lemma}[theorem]{Lemma}
\newtheorem{proposition}[theorem]{Proposition}
\newtheorem{corollary}[theorem]{Corollary}

\newcommand{\N}{\mathbb N}
\newcommand{\Z}{\mathbb Z}

\def\Z{{\mathbb Z}}

\usepackage[round]{natbib}

\author[Yann Bugeaud and Christophe Reutenauer]{Yann Bugeaud\affiliationmark{1,2} 
  \and Christophe Reutenauer \affiliationmark{3} }
\title[On the conjugates of Christoffel words]{On the conjugates of Christoffel words}
\affiliation{
 Universit\'e de Strasbourg et CNRS, Strasbourg, France\\
Institut universitaire de France\\
Universit\'e du Qu\'ebec \`a Montr\'eal, Montr\'eal, Canada}
\keywords{Combinatorics on words, Christoffel word, Ostrowski numeration}
\begin{document}
\publicationdata{vol. 27:3}{2025}{20}{10.46298/dmtcs.15140}{2025-01-27; 2025-01-27; 2025-06-17; 2025-10-21}{2025-10-23}
\maketitle

\begin{abstract} 
\vskip2mm
  We introduce a parametrization of the 
  conjugates of Christoffel words based on the integer Ostrowski numeration system. 
  We use it to give a precise description of the borders (prefixes which are also suffixes) of the 
  conjugates of Christoffel words and to revisit the notion of Sturmian graph 
  introduced by Epifanio et al. 
\end{abstract}


\section{Introduction}

In an article published in 1875 \cite {C}, Christoffel introduced a class of words on a binary alphabet, which now bear his name (following 
Berstel \cite{B}); they were shortly after rediscovered by Smith \cite{S}. Independently, in his 1880 work on minima of quadratic functions \cite{M1,M2}, Markoff 
used these words to construct certain quadratic forms, which satisfy sharp inequalities relating their minima and their discriminant. Markoff was certainly 
unaware of the work of Christoffel. The relation between the work of Markoff and the Christoffel words was explicitly known to Frobenius \cite{F}, 
who also formulated the famous conjecture on the Markoff numbers (see \cite{A}). The theory of these words may be found in several books: \cite{PF, L,BLRS,A,Re}.

The conjugates of Christoffel words (obtained by cyclically permuting these words) also have some importance, since they appear in different areas:

1. They coincide with 
the elements of the free group with two generators subject to the following conditions: they are
positive (that is, with no inverted letter), they are cyclically reduced, and they are part of a basis of this free group; see \cite{OZ, KR}. 

2. They are the ``perfectly clustering words" on a two-letter alphabet; this means that 
the last column of the Burrows-Wheeler tableau of 
such 
a word\footnote{This tableau was defined by Burrows and Wheeler in the theory of data compression, see \cite{RR} Section 4.}, whose rows are the lexicographically sorted conjugates, is decreasing; see \cite{MRS,FZ}. 

3. They constitute the finitary version of the Sturmian 
(infinite) words, which are obtained by discretizing straight lines in the plane, and which are characterized by the property that for each $n$, they have exactly $n+1$ factors of length $n$. For example, a word is a conjugate of a Christoffel word if and only if all its conjugates 
are factors of a Sturmian word; equivalently, this word (of length $n$ say) is primitive and has exactly $n-1$ circular factors of length $n-2$; see \cite{L}, \cite[Theorem 15.3.1]{Re}.

4. Besides the Christoffel words, which encode the Markoff forms and 
their minima, their conjugates correspond to the ``small values" of these quadratic forms; see \cite{Re1}. 

It is well known that Christoffel words are parametrized by nonnegative rational numbers. In the present article we first introduce a parametrization of the 
conjugates of Christoffel words, which is a finitary version of \cite{BL}. This parametrization is based on the integer Ostrowski numeration system. It generalizes 
a construction  which is widely used in the theory of Sturmian words, following Rauzy \cite{R} (the ``Rauzy rules"), and de Luca and Mignosi \cite{dL} (the 
``standard words"). The construction is given in (\ref{eq}). Theorem \ref{Vm} states 
that the whole conjugation class is constructed, and that it is independent of the chosen Ostrowski representation. As a corollary we obtain a result of Frid 
\cite{Fi}, which states (in some equivalent formulation) that the prefixes of a standard word are parametrized by legal Ostrowski representations; see Corollary 
\ref{frid}, which appears as a noncommutative lifting of the Ostrowski numeration system.

In Section \ref{BORDS}, we study the borders (a prefix which is also a suffix) of conjugates of Christoffel words. It is well known that the length of the longest 
border of a word and its smallest period are simply related: their sum is the length of the word. 
The study of periods in words is an important matter in combinatorics 
on words, in particular in the theory of Sturmian words: it is known that each finite Sturmian word has a nontrivial proper period, except precisely the Christoffel 
words. We thus focus on conjugates of Christoffel words, and determine their longest borders. Our parametrization of the conjugates allows us to give precise 
statements on the form of these borders. In particular, they are themselves conjugates, or a power of them (Theorem \ref{bords}, Corollaries \ref{power} and 
\ref{H}). Note that smallest periods of conjugates of Christoffel words have been previously computed by Lapointe \cite{La}, and applied by her to the determination of normal forms, thereby allowing her to characterize conjugates within the class of Sturmian words. The set of smallest periods is also studied in \cite{HN} and \cite{CS}.

In Section \ref{revisited}, we give an application of our methods to notions and results due to Epifanio, Frougny, Gabriele, Mignosi and Shallit \cite{EMSV,EFGMS}. The result of Frid, once formulated for 
the so-called ``lazy" Ostrowski representation (a notion introduced by these authors), may be translated into a result on the paths of a certain graph, called the ``compact graph"; it 
states that each suffix of the central word corresponding to a Christoffel word is the label of a unique path, starting from the origin, in this graph. By  specializing 
to lengths, one obtains the ``Sturmian graph": this graph has the property that each integer from 0 to the length of the central word is the label of a unique path; 
see Corollaries \ref{compact} and \ref{SturmianGraph}, due to \cite{EMSV,EFGMS}. As a consequence, we obtain the new result that these two graphs are 
naturally embedded in the tree of central words and in the Stern-Brocot tree; see Corollary \ref{embedded}, for the proof of which we use the ``iterated palindromisation" of 
Aldo de Luca.

Note that a link between lazy representations and periods of words was already established by Gabric, Rampersad and Shallit \cite{GRS}: they determine the set of 
periods of each prefix of length $n$ of a characteristic Sturmian (infinite) word and they show that the cardinality of this set is equal to the sum of the digits in the lazy representation of $n$.

In the next five (short) sections, we recall classical results on continuant polynomials, 
Ostrowski numeration, conjugation, and Christoffel words. Our new results are stated and proved 
in Sections \ref{construct} to \ref{revisited}.

\section{Continuant polynomials}\label{sec-cont}

{\it Continuant polynomials} are defined for any $k\geq 0$ and any integers $n_1,\ldots,n_k$ as follows:
$K_{-1}=0,K_0=1$ and  
$$     
K_k(n_1,\ldots,n_k)=K_{k-1}(n_1,\ldots,n_{k-1})n_k+K_{k-2} (n_1,\ldots,n_{k-2})
$$
for any $k\geq 1$ ({\it right recursion formula}). It is customary to drop the index $k$ and to write $K(n_1,\ldots,n_k)$ for $K_k(n_1,\ldots,n_k)$, and in particular $K()=1$.
One has (for example \cite{Co} p. 116)

\begin{equation}\label{continuant-matrice}
P(n_1)\cdots P(n_k)=\left( 
\begin{array} {cc}
K(n_1,\ldots,n_k)&K(n_1,\ldots,n_{k-1})\\
K(n_2,\ldots,n_{k})& K(n_2,\ldots,n_{k-1})
\end{array} 
\right),
\end{equation}
where $P(n)=\left( \begin{array} {cc}n&1\\1&0 \end{array} \right)$. By associativity of the matrix product, one obtains 
the {\it left recursion formula:} $$K(n_1,\ldots,n_k)=n_1K(n_2,\ldots,n_{k})+K(n_3,\ldots,n_{k}).$$ It follows also, by 
transposing the product, and using the symmetry of the matrices $P(n)$, 
that we have $K(n_1,\ldots,n_k)=K(n_k,\ldots,n_1)$.

For later use, we mention the identity, for $k\geq 1$,
\begin{equation}\label{n-1} K(n_1,\ldots,n_k)=K(n_1-1,n_2,\ldots,n_k)+K(n_2,\ldots,n_k),
\end{equation}
which follows easily from the left recursion formula.
The link with continued fractions is that each finite continued fraction 
$$
[n_1,\ldots,n_k] = n_1 + \frac{1}{n_2 + \frac{1}{n_3  + \cdots+\frac{1}{n_k}} }
$$
is equal to the reduced fraction $K(n_1,\ldots,n_k)/K(n_2,\ldots,n_k)$. Equivalently, the continued fraction $[0,n_1,\ldots, n_k]$ is equal to $K(n_2,\ldots,n_k)/K(n_1,\ldots,n_k)$.

\section{Ostrowski numeration}\label{Onumer}


Let $a_1,a_2, \ldots, a_m$ be a finite sequence of positive natural numbers. Define the positive integers $q_0,q_1,\ldots,$ $q_{m}$ by $q_i=K(a_1,\ldots,a_i)$, 
for $i=0, \ldots , m$. 
Note that 
$q_0=1$, $q_1=a_1$, $q_2=a_1a_2+1$, and we let $q_{-1}=0$, in accordance with the conventions for continuant polynomials. Note that the right recursion for continuant polynomials gives $q_i=q_{i-1}a_i+q_{i-2}$ for any $i = 1, \ldots ,  m$. Note that the sequence of $q_i$, $i\geq -1$, is strictly increasing, except for the following case: $a_1=1$, $q_0=q_1$.

It is useful to define
\begin{equation}
b_1=a_1-1, b_i=a_i \,\, \mbox{if} \,\, i\geq 2.
\end{equation}

Any expression
\begin{equation}\label{Ost1}
N=d_1q_0+d_2q_1+\cdots+d_mq_{m-1},
\end{equation}
where the digits $d_i$'s are in $\mathbb Z$, is called {\it (unrestricted) Ostrowski representation} of the integer $N$. 
We stress that, unlike in previous works, we allow the digits to be negative.

The representation (\ref{Ost1}) is called {\it legal} if one has the inequalities
 \begin{equation}\label{ineq}
\forall i\geq 1, 0\leq d_i\leq b_i.
\end{equation}
Among the legal representations, we distinguish two of them. We say that 
the representation (\ref{Ost1}) is {\it greedy} if it is legal and if the following condition is satisfied
\begin{equation}\label{greedy}
\forall i\geq 2, d_i=b_i \Rightarrow d_{i-1}=0.
\end{equation}

We say that the representation (\ref{Ost1}) is {\it lazy} if it is legal and if, with $k=\max\{i\mid d_i\neq 0\}$,
\begin{equation}\label{lazy}
\forall i, 2\leq i\leq k, d_i=0 \Rightarrow d_{i-1}=b_{i-1}.
\end{equation}

\begin{proposition}\label{greedylazy} (i) Each integer $N=0,\ldots,q_m-1$ has a unique greedy representation.

(ii) Each $N=0,\ldots,q_m+q_{m-1}-2$ has a unique lazy representation.
\end{proposition}

The existence of a representation (\ref{Ost1}) is implicit in Ostrowski's article \cite[p.178]{O}; (i) is stated in \cite{D} p.83, 
and proved by Fraenkel, \cite[Theorem 3]{Fr} (see also \cite[Theorem 3.9.1]{AS} for a proof).
Lazy Ostrowski representations were introduced by Epifanio, Frougny, Gabriele, Mignosi and Shallit in \cite{EFGMS}; (ii) follows from their work.

For the sake of completeness, we give a proof of Proposition \ref{greedylazy} in the Appendix (Section \ref{app}).

For later use, we state the following result. We say that a sequence $d_1,\ldots,d_k$ is {\it alternating} if its values are alternatively $0$ and $b_i$; there are therefore two alternating 
sequences of length $k$.

\begin{lemma}\label{Ost-C} Let $\sum_{j=1}^{m} d_j q_{j-1}$ be a greedy Ostrowski representation. 
Then, the inequality 
$$
\sum_{j=1}^{m} (b_j -d_j)q_{j-1}\leq q_{m-1}-1
$$
holds if and only if $d_m = b_m$ and the sequence $d_i, i=1,\ldots,m$, is alternating.
\end{lemma}

\begin{proof} By Proposition \ref{greedylazy} (ii), $\sum_{j=1}^{m} b_jq_{j-1}=q_m+q_{m-1}-2$ since the left-hand side is a lazy representation, and is necessarily the largest 
one. Thus the inequality of the lemma is equivalent to $\sum_{j=1}^{m} d_jq_{j-1}\geq q_m-1$. By the proposition 
again, part (i) this time, this inequality is equivalent to the fact
that
the left-hand side is the unique greedy representation of $q_m-1$. 
But by Lemma \ref{altlazyineq} (i), this unique representation is the alternating one, with $d_m=b_m$.
\end{proof}

\section{Conjugation}

We consider an alphabet $A$, the free monoid $A^*$ generated by $A$ and the free group $F(A)$ generated by $A$. Let $1$ denote the identity element of $A^*$.
If $g$ is in $F(A)$ and $x$ in $A$, we 
denote by $|g|_x$ the number of occurrences of $x$ in $g$, where one counts with -1 the occurrences of $x^{-1}$; this is well defined and does not depend on the expression for $g$. Moreover, define $|g|=\sum_{x\in A}|g|_x$, the {\it algebraic length} of $g$. In particular, if $g\in A^*$, then $|g|$ is the 
{\it length} of $g$.

Two words $u,v$  in $A^*$ are called {\it conjugate} if  for some words $x,y\in A^*$, one has $u=xy,v=yx$. 
The {\it conjugator} is the mapping of $A^*$ into itself that maps each word $w=au$, $a\in A,u\in A^*$, onto $ua$ (with $C(1)=1$). Hence two words in $A^*$ are conjugate if and only one is the image of the other under some power of the conjugator: $v=C^{|x|}(u)$, with the previous notations. 

Since $yx=x^{-1}(xy)x$, two words $u,v$ conjugate in $A^*$ are conjugate in $F(A)$, too. The converse is also true, as is well known, and one may be more precise.

\begin{lemma}\label{conjug} Let $u,v\in A^*$, $g\in F(A)$ be such that $v=g^{-1}ug$. Then $u,v$ have the same length $n$ and $v=C^{|g|}(u)$. Let $r$ be the remainder of the Euclidean division of $|g|$ by $n$. Then $u=xy,v=yx$, $u,v\in A^*$, with $x$ of length $r$.
\end{lemma}

\begin{proof} The first assertion is clear, by definition of the algebraic length. 

We may assume that $g$ is reduced, that is, that $g$ is written as a product of elements of $A$ and their inverses, in such a way that no factor $aa^{-1}$ nor $a^{-1}a$ occurs in this 
product (one obtains a reduced expression of an element $g$ by 
removing
these factors; it is well known that this algorithm does not change the algebraic length $|g|$).

We show that $v=C^{|g|}(u)$, by induction on the length of the reduced expression of $g$. If this length is 0, 
then $g=1$ and the result is evident. Suppose that the length of $g$ is $\geq 1$. We 
have $v=g^{-1}ug$ and $v$ is reduced, being in $A^*$. Hence the first letter $a$ of $u$ is equal to the inverse of the last letter of $g^{-1}$, that is, equal to the first letter of $g$. 
Thus $u=au_1, g=ag_1$, $u_1\in A^*$, $g_1\in 
F(A)$, and $g_1$ is reduced and its length is one less than that of $g$. Then $v=(ag_1)^{-1}au_1ag_1=g_1^{-1}u_1ag_1$. By induction, $v=C^{|g_1|}(u_1a)$. Hence $g=C^{|g_1|}\circ C(u)=C^{|g_1|+1}(u)$, which implies the result.

Since $C^n$ is the identity on the words of length $n$, we have $C^{|g|}(u)=C^r(u)$, and this implies the last assertion.
\end{proof}

\section{Morphisms}\label{morph}

We consider now the alphabet $A=\{a,b\}$ ordered by $a<b$.

The endomorphism of $A^*$ (resp. $F(A)$), sending $a$ onto $u$ and $b$ onto $v$, is denoted 
by $(u,v)$. Each endomorphism of $A^*$ extends uniquely to an endomorphism of $F(A)$.

We define certain endomorphisms of $A^*$ and $F(A)$: $$E=(b,a), G=(a,ab), \widetilde G=(a,ba), D=(ba,b), \widetilde D=(ab,b),$$ and 
$$
\pi(i,j)=(a^iba^j,a),
$$ 
for all nonnegative integers $i,j$. 
Note that all these endomorphisms, when viewed on $F(A)$, are automorphisms of $F(A)$. 

One has $G^i=(a,a^ib), D^i=(b^ia,b)$, and $\widetilde G^j=(a,ba^j), \widetilde D^j=(ab^j,b)$ for all 
nonnegative integers $i,j$. It follows that
\begin{equation}\label{piGE}
\pi(i,0)=G^iE=ED^i, \pi(0,j)=E\widetilde D^j=\widetilde G^jE.
\end{equation}
In particular, the involution $E$ conjugates $G,D$, and $\widetilde G,\widetilde D$.

Given an endomorphism $f$ of $F(A)$, its {\it abelianization} is the matrix 
$$
M(f)=\left( \begin{array} {cc}|f(a)|_a&|f(b)|_a\\|f(a)|_b&|f(b)|_b \end{array} \right).
$$ 
This function is multiplicative: $M(f'\circ f)=M(f')M(f)$, for any other endomorphism $f'$.
One has for any element $g\in F(A)$,
\begin{equation}\label{calcul}
\left( \begin{array} {cc}|f(g)|_a\\|f(g)|_b \end{array} \right)=M(f) 
\left( \begin{array} {cc}|g|_a\\|g|_b \end{array} \right).
\end{equation}
Observe that the abelianization of the endomorphism $(a^iba^j,a)$ is given by
\begin{equation}\label{abel}
M((a^iba^j,a))=P(i+j), 
\end{equation}
where $P$ is defined in Section \ref{sec-cont}.

For $g\in G$ ($G$ is here a group), we denote by $\gamma(g)$ the conjugation by $g$: $$\gamma(g)(x)=gxg^{-1}.$$ One has $\gamma(gh)=\gamma(g)\circ\gamma(h)$. For later use, we state the following lemma (which is related to the well-known result that the subgroup of inner automorphisms of $G$ is a normal subgroup of the group of all automorphisms of $G$).

\begin{lemma}\label{phipsi} Let $\varphi_i,\psi_i$, $i=1,\ldots,m$, be automorphisms of a group $G$ and $g_1,\ldots,g_m\in G$ be such that $\varphi_i=\gamma(g_i)\psi_i$. Then
$$
\varphi_1\cdots\varphi_m=\gamma(g)\psi_1\cdots\psi_m,
$$
where 
$$
g=g_1\psi_1(g_2)\cdots(\psi_1\cdots\psi_{m-1})(g_m).
$$
\end{lemma}

The proof, by induction on $m$, is left to the reader.

\section{Christoffel words}\label{Cword0}

Among many equivalent definitions of Christoffel words, we choose one that is useful for our purpose. A {\it lower} (resp. {\it upper}) {\it Christoffel word} is the 
image of 
$a$ or $b$ under an endomorphism of $A^*$ belonging to the monoid of endomorphisms generated by $G$ and $\widetilde D$ (resp. $\widetilde G$ and $D$). A {\it Christoffel word} is 
a lower or an upper Christoffel word. It follows from these definitions that the 
endomorphisms 
$G,\tilde D$ (resp. $\tilde G,D$) preserve lower (resp. upper) Christoffel words.

The conjugation class of some Christoffel word is called a {\it Christoffel class}. It is known that in each Christoffel class, there is exactly one lower, and one upper, Christoffel word. See \cite{L,R} for this and other properties of these words.

Since the involution $E$ exchanges $a$ and $b$, and conjugates $G$ and $D$ (resp. $\widetilde G$ and $\widetilde D$), it exchanges lower and upper Christoffel words.

We define two rational numbers associated to a word $w$. We call {\it Slope} of $w$ the ratio 
$|w|_b/|w|$, 
and {\it slope} of $w$ the ratio $|w|_b/|w|_a$ (it is infinite if 
$w=b$). It follows from the general theory of Christoffel words that for each $s$ 
in $\mathbb Q_+\cup\infty$ (resp. each $S$ in  $[0,1]$), there exists a unique lower (resp. upper) Christoffel 
word of slope $s$ (resp. of Slope $S$); for $s\neq 0,\infty$ (resp. $S\neq 0,1$), these two Christoffel words are distinct and conjugate. 

Denoting by $S$ and $s$ the Slope and the slope respectively, one has
$$
S=\frac{s}{1+s}, \,\, s=\frac{S}{1-S}.
$$
Equivalently, $S^{-1}=1+s^{-1}$. We have $S=0$ if and only if $s=0$, and $S=1$ if and only if $s=\infty$. Otherwise, $0<S<1$, and the continued fraction of $S$ is of the form $[0,a_1,\ldots,a_m]$, where the $a_i$ are positive integers. Then $s^{-1}=S^{-1}-1=[a_1,\ldots,a_m]-1=[a_1-1,a_2,\ldots,a_m]$ if $a_1\geq 2$, and therefore $s=[0,a_1-1,a_2,\ldots,a_m]$; and if $a_1=1$, we have $s^{-1}=[0,a_2,\ldots,a_m]$ hence $s=[a_2,\ldots,a_m]$.

\section{Construction of the conjugates of a Christoffel word}\label{construct}

We fix 
a sequence $a_1,\ldots,a_m$ of positive integers and define $b_i,q_i$ as in Section \ref{Onumer}.

Following \cite{BL}, 
given a sequence of integers $d_1,\ldots,d_m$ in $\Z$, 
we define the following sequence $V_i=V_i(d_1,\ldots,d_m)$, of elements of $F(A)$, by
$$
V_{-1}=b,V_0=a,
$$
and, for $i=1,\ldots,m$, 
\begin{equation}\label{eq}
V_i=V_{i-1}^{b_i-d_i}V_{i-2}V_{i-1}^{d_i}.
\end{equation}

Note that we do not ask for the moment that the $d_i$ be nonnegative. This implies that the exponents in the previous equations may be negative, and the $V_i$ may be in $F(A)\setminus A^*$.

It is useful to note that one has the following {\it stability property}: $V_i(d_1,\ldots,d_m)$ depends only on $a_1,\ldots,a_i$ and on $d_1,\ldots, d_i$. Note that the lengths of
the words $V_i$, $i\geq 1$, are strictly increasing; the lengths of $V_{-1}=b$ and $V_0=a$ are 1, and the length of $V_1$ is 1 exactly when $a_1=1$, in which case $V_1=b$; if $a_1>1$, then $|V_0|=1<|V_1|$.

\begin{lemma}\label{endoV} With the previous definition, for any $i=0,\ldots,m$, the endomorphism $(V_i,V_{i-1})$ is equal to $$
\pi(b_1-d_1,d_1)\circ\cdots\circ\pi(b_i-d_i,d_i). 
$$
In particular, the words $V_{i-1}$ and $V_i$ form the basis of a free submonoid of $\{a,b\}^*$.
\end{lemma}

\begin{proof} 
The morphism
$(V_0,V_{-1})=(a,b)$ is the identity morphism, so that the formula is true for $i=0$. Let $i\geq 1$; then $$(V_i,V_{i-1})=(V_{i-1}^{b_i-d_i}V_{i-2}V_{i-1}^{d_i},V_{i-1})=
(V_{i-1},V_{i-2})\circ (a^{b_i-d_i}ba^{d_i},a)$$$$=(V_{i-1},V_{i-2})\circ \pi(b_i-d_i,d_i).$$ By induction on $i$, we have $$(V_{i-1},V_{i-2})=\pi(b_1-d_1,d_1)\circ\cdots\circ\pi(b_{i-1}-d_{i-1},d_{i-1}).$$ Thus the first assertion of the lemma follows.

The last one
follows from the injectivity of the morphisms $\pi(i,j)$ (because they extend to automorphisms of the free group), hence of their product.
\end{proof}

\begin{lemma}\label{length} Let $V_m=V_m(d_1,\ldots,d_m)$.

(i) $|V_m|_a=K(a_1-1,a_2,\ldots,a_m)$, $|V_m|_b=K(a_2,\ldots,a_m)$, $|V_m|=K(a_1,\ldots,a_m)$.

(ii) The Slope of $V_m$ is $[0,a_1,\ldots,a_m]$.
\end{lemma}

\begin{proof} We have, by Lemma \ref{endoV}, $$V_m=(V_m,V_{m-1})(a)=\pi(b_1-d_1,d_1)\circ\cdots\circ\pi(b_m-d_m,d_m)(a).$$ It follows from Section \ref{morph} that 
$$\left( \begin{array} {cc}|V_m|_a\\|V_m|_b \end{array} \right)=P(b_1)\cdots P(b_m)
\left( \begin{array} {cc}1\\0 \end{array} \right).$$ Thus by (\ref{continuant-matrice}) $|V_m|_a=K(b_1,\ldots,b_m)$ and $|V_m|_b=K(b_2,\ldots,b_m)$. We have $b_i=a_i$, except for $i=1$, where $b_1=a_1-1$. Thus (i) follows, using (\ref{n-1}) for the third formula, and (ii) follows at once.
%
\end{proof}

If the sequence $d_1,\ldots,d_m$ satisfies the inequalities (\ref{ineq}), then the exponents in (\ref{eq}) are all nonnegative, therefore $V_i(d_1,\ldots,d_m)\in A^*$. 
Define 
\begin{equation}\label{Mm}
M_m= V_m(0,\ldots,0)=\pi(b_1,0)\circ\cdots\circ\pi(b_m,0)(a),
\end{equation} 
the second equality holding by Lemma \ref{endoV}. Then $M_m\in A^*$, and $M_m$ is of length $q_m=K(a_1,\ldots,a_m)$, by Lemma \ref{length}.

\begin{theorem}\label{Vm} The element $V_m=V_m(d_1,...,d_m)$ is conjugate within $F(A)$ to $M_m$. Precisely, $V_m=h^{-1}M_mh$ for some $h\in F(A)$ of algebraic length $N=d_1q_0+
\cdots+d_mq_{m-1}$.

The $A^*$-conjugation class of $M_m$ is equal to the set of all $V_m(d_1,\ldots,d_m)$, for all sequences $d_1,\ldots,d_m$ satisfying (\ref{ineq}) and precisely $V_m=C^N(M_m)$, 
with $N$ as above.
This class contains the two Christoffel words of Slope $S=[0,a_1,\ldots,a_m]$.
\end{theorem}

It follows from this theorem that to each sequence $a_1,\ldots,a_m$ of positive integers, we associate a Christoffel class.

\begin{proof} By Lemma \ref{endoV}, $$V_m=(V_m,V_{m-1})(a)=\pi(b_1-d_1,d_1)\circ\cdots\circ\pi(b_m-d_m,d_m)(a).$$ We have $$\pi(i,j)=(a^iba^j,a)=\gamma(a^{-j})\circ (a^{i+j}
b,a)=\gamma(a^{-j})\circ \pi(i+j,0).$$ We apply Lemma \ref{phipsi} with $\varphi_i=\pi(b_i-d_i,d_i), \psi_i=\pi(b_i,0), g_i=a^{-d_i}$. We obtain that
$$V_m=\gamma(g)\circ\pi(b_1,0)\circ\cdots\circ\pi(b_m,0)(a)=\gamma(g)(M_m),$$ where $g$ is equal to
\begin{equation}\label{g}
a^{-d_1}(\pi(b_1,0)(a^{-d_2}))\cdots(\pi(b_1,0)\circ\cdots\circ\pi(b_{m-1},0)(a^{-d_m})).
\end{equation}
This implies that $V_m$ is conjugate within $F(A)$ to $M_m$. 

Let $h$ be the inverse of $g$. Then 
\begin{equation}\label{MmhVm}
V_m=h^{-1}M_mh
\end{equation} 
and, by (\ref{abel}) and (\ref{calcul}), the algebraic length of $h$ is equal to
$$
d_1+(1,1)P(b_1)\,{}^t\!(d_2,0)+\cdots+(1,1)P(b_1)\cdots P(b_{m-1})\,{}^t\!(d_m,0).
$$
By (\ref{continuant-matrice}), this is
$$
d_1+(K(b_1)+K())d_2+\cdots+(K(b_1,\ldots,b_{m-1})+K(b_2,\ldots,b_{m-1}))d_m$$$$=d_1q_0+d_2q_1+\cdots+d_mq_{m-1}=N, 
$$
by (\ref{n-1}), since $b_i=a_i$ if $i\geq 2$, and $b_1=a_1-1$.

If the sequence $d_1,\ldots,d_m$ satisfies (\ref{ineq}), then $V_m$ is in $A^*$, and by Lemma \ref{conjug}, $V_m=C^{N}(M_m)$, thus $V_m$ is in the conjugation class of $M_m$. 
Conversely, each element of this class appears, since $M_m$ is of length $K(a_1,\ldots,a_m)$ (Lemma \ref{length}), and since, by Proposition \ref{greedylazy}, each 
$N=0,\ldots,K(a_1,\ldots ,a_m)-1$ has an Ostrowski representation 
satisfying (\ref{ineq}).

We show now that the class contains a Christoffel word. Consider the sequence $d_1,\ldots, d_m$ defined by $d_m=b_m,d_{m-1}=0,d_{m-2}=b_{m-2}$, and so on, depending on 
the parity of $m$. The corresponding element $V_m$ is in 
$A^*$, and is equal to $\cdots \pi(0,b_{m-2})\circ\pi(b_{m-1},0)\circ\pi(0,b_m)(a)$. 
If $m$ is even, then, by (\ref{piGE}), we have 
$$V_m=G^{b_{1}}EE\widetilde D^{b_{2}}\cdots G^{b_{m-1}}EE\widetilde D^{b_m}(a);$$ since $E$ is an involution, since $a$ is a lower Christoffel word, and since $\widetilde D,G$ preserve lower Christoffel words, we obtain that $V_m$ 
is a lower Christoffel word.
If $m$ is odd, then similarly $$V_m=\widetilde G^{b_1}EED^{b_2}\cdots \widetilde G^{b_m}E(a);$$ since $E(a)=b$ is an upper Christoffel word, and since $D,\widetilde G$ preserve upper Christoffel words, we 
obtain that $V_m$ is an upper Christoffel word.
\end{proof}

\begin{corollary}\label{Cword} Let $N=\sum_{1\leq i\leq m}d_iq_{i-1}$ be a greedy representation. Then $V_m(d_1,\ldots,d_m)$ is a Christoffel word if and only if the sequence $d_1,\ldots,d_m$ is alternating. Said more precisely, for $m\geq 1$, the word
$V_m (b_1, 0, b_3, 0, \ldots)$ is an upper Christoffel word, and 
the word
$V_m (0, b_2, 0, b_4, \ldots)$ is a lower Christoffel word. 
\end{corollary} 

\begin{proof} We know that each conjugation class of Christoffel word contains exactly one lower, and one upper, Christoffel word. By Theorem \ref{Vm} and 
Proposition \ref{greedylazy} (i), the mapping $N\mapsto V_m(d_1,\ldots,d_m)$, where $N=\sum_{1\leq i\leq m}d_iq_{i-1}$ is the greedy representation of $N$, is a 
bijection from $\{0,1,\ldots,q_m-1\}$ onto the conjugation class of $M_m$.
Thus,
it is enough to 
show the last assertion. By the end of the proof of Theorem \ref{Vm}, the two indicated words are Christoffel words. Note that a Christoffel word, distinct from $a,b$ (which are both lower and upper), is a lower one if and only if it 
begins by $a$.
We observe that if $V_i$ and $V_{i+1}$ begin by some letter $x$, then so 
do all the words $V_j$ for $j\geq i$. 

Consider the alternating sequence beginning by $0$, namely: $d_1=0,d_2=b_2,\ldots$. Then $V_1=a^{b_1}b$ begins by $a$ if $a_1\geq 2$, and is equal to $b$ if $a_1=1$. Thus, if $a_1\geq 2$, 
then $V_0,V_1$ begin by $a$, hence also do all $V_i$, $i\geq 0$. If $a_1=1$, then $V_2=V_0V_1^{b_2}=ab^{b_2}$, and $V_3=V_2^{b_3}V_1$ both begin by $a$; hence $V_i$ 
begins by $a$ for $i\geq 2$.

Consider now the alternating sequence beginning by $b_1$, namely: $d_1=b_1,d_2=0,\ldots$. Then $V_1=ba^{b_1},V_2=V_1^{b_2}a$ both begin by $b$, and therefore all $V_i$, $i\geq 1$, begin by $b$.
\end{proof}

Throughout the paper ${\widetilde x}$ denotes the reversal (mirror image) of 
the word $x$.
A {\it palindrome} is a word equal to its reversal. The empty word is a palindrome.
Recall that each proper lower Christoffel word $w$ has the factorization $w=apb$, where $p$ is a palindrome (called a {\it central word}), and that then the 
corresponding upper Christoffel word is $\widetilde w=bpa$, which is a conjugate of $w$. A {\it standard word} is $a$ or $b$, or a word of the from $pab$ or $pba$ for some central word $p$; it is known that standard words are obtained from $a,b$ by applying the endomorphisms in the submonoid generated by $G$ and $D$; moreover, each Christoffel class contains exactly two standard words $pab$ and $pba$, where $p$ is the corresponding central palindrome. See \cite[Subsection 2.2.1]{L}, \cite{Re}.

\begin{corollary}\label{standard}
For $m \ge 1$, the word $M_m$ is a standard word, 
equal to $pab$ if $m$ is odd and to $pba$ if $m$ is even.    
\end{corollary}

\begin{proof} An easy induction shows that 
$$M_m=\pi(b_1,0)\circ\cdots\circ\pi(b_m,0)(a)$$ ends 
by $b$ if $m$ is odd, and by $a$ if $m$ is even. 

Suppose that $m$ is even. Then by (\ref{piGE}),
$$M_m=\pi(b_1,0)\circ\cdots\circ\pi(b_m,0)(a)=G^{b_1}EED^{b_2}\cdots ED^{b_m}(a).$$ Hence $M_m$ is a standard word. Since it ends by 
$a$, we have $M_m=pba$. 

Suppose that $m$ is odd. Then $$M_m=\pi(b_1,0)\circ\cdots\circ\pi(b_m,0)(a)=G^{b_1}EED^{b_2}\cdots ED^{b_{m-1}}G^{b_m}E(a).$$ Hence $M_m$ is a standard 
word. Since it ends by $b$, we have $M_m=pab$. 
\end{proof}

%
%
%
%
%
%
%
%

We may derive a result, which is equivalent to a result previously obtained by Frid, \cite{Fi} Corollary 1, and which is a noncommutative version of the Ostrowski representation.

We consider below the Christoffel class associated to the sequence $a_1,\ldots,a_m$.

\begin{corollary}\label{frid} Let $N= 0, \ldots , q_m-2$ be an integer  
whose legal Ostrowski representation is given by $N=\sum_{1\leq i\leq m}d_iq_{i-1}$. 
Then the prefix of length $N$ of the central palindrome $p$ is 
$$
M_{m-1}^{d_m}\cdots M_0^{d_1}.
$$
In particular this product depends only on $N$ and not on the chosen legal Ostrowski representation of $N$.
\end{corollary}

\begin{proof} 
The element $h=g^{-1}$, appearing in the proof of Theorem \ref{Vm}, is by 
(\ref{g}) and (\ref{Mm}) equal to
$$
h=(\pi(b_1,0)\circ\cdots\circ\pi(b_{m-1},0)(a^{d_m})) \cdots (\pi(b_1,0)(a^{d_2})) a^{d_1} 
$$
$$
=M_{m-1}^{d_m} \cdots M_{1}^{d_2} M_0^{d_0}.
$$
In particular, $h$ is in $A^*$. Because of the inequalities (\ref{ineq}), the    
word $V_m$ is in $A^*$, and $M_m$ is in $A^*$ too. 
Moreover, $h$ is of length $N$, by a calculation in the proof of Theorem \ref{Vm}; hence $|h|<q_m=|M_m|$. 
Thus, by (\ref{MmhVm}) and  by Lemma \ref{conjug}, $h$ is a prefix of $M_m$.
Since $M_m=pab$ or $M_m=pba$ is of length $q_m$, we get that $h$ is a prefix of $p$. 
\end{proof}

Define the sequence $c_i,i=1,\ldots,m$, by $c_i=a_i$ for $i=2,\ldots,m-1$ and $c_i=a_i-1$ for $i=1,m$; in other words, the $c_i$ coincide with the $a_i$, except the two extremes $c_1,c_m$, which are one less; note that $c_i=b_i$, except that $c_m=b_m-1$, if $m\geq 2$. For later use, we prove

\begin{lemma}\label{pal} Let $1\leq i\leq m$ and let $0\leq c\leq c_i$. The word
$$
M_{i-1}^{c} M_{i-2}^{c_{i-1}}\cdots M_{1}^{c_2} M_0^{c_1}
$$
is a palindrome. 
\end{lemma}

This lemma could be deduced from a result of de Luca and Mignosi \cite[Prop. 7]{dLM}.  
We give an independent proof. See also Lemma \ref{XWpal} below.

\begin{proof} By stability, it is enough to prove this result for $i=m$. Suppose first that $b_1=c_1\geq 1$. 
We have by Lemma \ref{endoV} and (\ref{piGE}),
$$
M_{m-1}^{c} M_{m-2}^{c_{m-1}}\cdots M_{1}^{c_2} M_0^{c_1}$$
$$=[\pi(b_1,0)\cdots\pi(b_{m-1},0)(a^c)][\pi(b_1,0)\cdots\pi(b_{m-2},0)(a^{c_{m-1}})]$$
$$\cdots[\pi(b_1,0)(a^{c_2})][a^{c_1}]
$$
$$=[G\pi(b_1-1,0)\cdots\pi(b_{m-1},0)(a^c)][G\pi(b_1-1,0)\cdots\pi(b_{m-2},0)(a^{c_{m-1}})]$$
$$\cdots[G\pi(b_1-1,0)(a^{c_2})][G(a^{c_1-1})]a
=G(u)a
$$
where 
$$
u=[\pi(b_1-1,0)\cdots\pi(b_{m-1},0)(a^c)][\pi(b_1-1,0)\cdots\pi(b_{m-2},0)(a^{c_{m-1}})]$$
$$\cdots[\pi(b_1-1,0)(a^{c_2})][a^{b_1-1}].
$$
By induction on the sum of the $a_i$, $u$ is a palindrome. Hence $G(u)a$ is a palindrome 
\cite[Lemma 4.1.4]{Re}.

Suppose now that $b_1=0$, that is $a_1=1$. Then, since $V_0=a,V_1=b$, the sequence of words $V_i, i=1,\ldots,m,$ is obtained from a shorter sequence, to which one applies $E$. We may therefore conclude by induction on $m$. 
\end{proof}

The next result is of independent interest. Before stating it, we point out that if $d_1, \ldots, d_m$ 
is a greedy representation, then $b_1 - d_1, \ldots, b_m - d_m$ is a lazy representation.

\begin{proposition}\label{mirror}
For any $d_1, \ldots , d_m$, with $0 \le d_k \le b_k$ for $k = 1, \ldots , m$,
the word $V_m (d_1, \ldots , d_m)$ is the mirror image of the word 
$V_m (b_1 - d_1, \ldots , b_m - d_m)$. 
\end{proposition}

\begin{proof}
This is proved by induction. Write
$$
V_m = V_{m-1}^{b_m - d_m} V_{m-2} V_{m-1}^{d_m}.
$$
Then, we have
$$
\widetilde V_m = \widetilde V_{m-1}^{d_m} \oV_{m-2} \oV_{m-1}^{b_m - d_m} = V_m (b_1 - d_1, \ldots , b_m - d_m),
$$
since, by the induction hypothesis, we have 
$$
\oV_{m-2} = V_{m-2} (b_1 - d_1, \ldots , b_{m-2} - d_{m-2}), 
$$
and
$$
\oV_{m-1} = V_{m-1} (b_1 - d_1, \ldots , b_{m-1} - d_{m-1}). 
$$
The proof is complete. 
\end{proof}

As a consequence of Proposition \ref{mirror}, we find the following well-known result.    

\begin{corollary}
Each Christoffel class  
comprises at most one palindrome and 
it comprises one palindrome precisely when $b_1, \ldots , b_m$ are all even, that is precisely 
when the words in the class have odd length. 
\end{corollary}

This result is not new: it is a consequence for example of the fact that the Burrows-Wheeler tableau 
of a Christoffel word (and even each perfectly clustering word) has a central symmetry (Theorem 4.3 of Simpson and Puglisi \cite{SP}).

\section{Borders of conjugates of Christoffel words}\label{BORDS}

We keep the notation of Section \ref{Onumer}. Recall that a {\it border} of a word is a nontrivial proper prefix which is also a suffix of this word.   

\subsection{Borders} 

In this subsection, we determine the longest border of every conjugate of a Christoffel word, thereby 
reproving a result of Lapointe \cite{La} 
(but with a totally different method). 
Indeed, the length of the longest border and the smallest nontrivial period of a word are related: their sum is the length of the word.

Before stating the main result of this section, recall that Corollary \ref{Cword} characterizes the cases where $V_m$ is a Christoffel word: informally speaking, the digits $d_i$ 
of the greedy representation must alternate between $b_i$ and 0. This extends by stability to each word $V_i$, $i<m$. It is well known that a Christoffel word has no border, 
which explains the hypothesis in the next result.

Moreover, in this result, we give the longest border of $V_m$. The other borders are all determined using Lemma \ref{bordmax}. 
\begin{theorem} \label{bords} 
Suppose that $m\geq 3$ or $m=2$ and $b_1 \ge 1$.
Let $N$ be an integer with $0 \le N \le q_m - 1$ and 
$N = \sum_{1 \le i \le m} d_i q_{i-1}$ be its greedy representation. 
Put $V_m = V_m (d_1, \ldots , d_m)$, $V_{m-1} = V_{m-1} (d_1, \ldots , d_{m-1})$, 
and $V_{m-2} = V_{m-2} (d_1, \ldots , d_{m-2})$. Assume that $V_m$ is not a Christoffel word. 
Let 
$$
\ell=\min\{ b_m - d_m, d_m\} \quad \hbox{and} \quad h = \min\{b_{m-1} - d_{m-1}, d_{m-1} + 1\}.
$$
Let $B$ be the longest border of $V_m$.

\begin{enumerate}[(i)]
\item If $d_m = b_m$, then 
$B=V_{m-1}$.

\item If $1 \le d_m \le b_m - 1$ and $1 \le d_{m-1} \le b_{m-1} - 1$, then $B=V_{m-1}^\ell$.

\item If $1 \le d_m \le b_m - 1$ and $d_{m-1} = 0$, then $B=V_{m-1}^{\ell}$, except if $b_m-d_m < d_m$ and  the sequence $d_{1},\ldots,d_{m-1}$ is not alternating, in 
which case $B=V_{m-1}^{\ell+1}$.

\item If $1 \le d_m \le b_m - 1$ and $d_{m-1} = b_{m-1}$, then $B=V_{m-1}^\ell$.

\item If $d_m=0$ and $b_m\geq 2$, then $B=V_{m-1}^{b_{m}-1}V_{m-2}$.

\item If $d_m=0$, $b_m=1$, and $b_{m-1} - d_{m-1} \ge 1$, $B=V_{m-2}^h$, 
except if $m \ge 3$,
$d_{m-2}=0$, $b_{m-1}-d_{m-1}< d_{m-1}+1$ and the sequence $d_1,\ldots,d_{m-2}$ is not alternating, 
in which case $B=V_{m-2}^{h+1}$.

\item If $d_m=0$, $b_m=1$, and $d_{m-1}=b_{m-1}$, then $B=V_{m-2}$.
\end{enumerate}
\end{theorem}

Let us comment briefly the theorem. 
Since $V_m = V_{m-1}^{b_m - d_m} V_{m-2} V_{m-1}^{d_m}$, we see that 
$V_{m-1}^{\min\{ b_m - d_m, d_m\}}$ is an obvious border of $V_m$. 
The point is that it may happen that it is not the longest.   
Indeed, if the last three digits in the greedy representation of $N$ are $d_{m-2}, 0, d_m$, 
with $d_m \ge 1$ and $d_{m-2} \le b_{m-2} - 1$, then 
$$
\Bigl( \, \sum_{1 \le i \le m-3} d_i q_{i-1} \Bigr) 
+  (d_{m-2} + 1) q_{m-3} + b_{m-1} q_{m-2} + (d_m - 1) q_{m-1} 
$$
is a legal 
representation of $N$. These representations induce, respectively, the factorizations
\begin{equation}\label{eqm}
V_m = V_{m-1}^{b_m - d_m} V_{m-2} V_{m-1}^{d_m}
\end{equation}
and
$$
V_m = V_{m-1}^{b_m - d_m + 1} V'_{m-2} V_{m-1}^{d_m - 1},
$$
where we have $V_{m-1} V'_{m-2} = V_{m-2} V_{m-1}$ and 
$V_{m-1} = V_{m-2}^{b_{m-1}} V_{m-3} = V_{m-3} (V'_{m-2})^{b_{m-1}}$. 
In this case, $V_{m-1}^{\min\{ b_m - d_m + 1, d_m\}}$ is the longest border of $V_m$. 

The key point for the proof of Theorem \ref{bords} 
is the determination of all the occurrences of $V_{m-1}$ in $V_m$. 
Exactly $b_m$ of them can be read on the factorization (\ref{eqm}), but there may be 
additional ones. By primitivity of $V_{m-1}$, the word $V_{m-1} V_{m-1}$ 
contains exactly two occurrences of $V_{m-1}$. Consequently, if an additional 
occurrence of $V_{m-1}$ appears, then it must be a factor of $V_{m-1} V_{m-2} V_{m-1}$. 
A more precise statement is given in Lemma \ref{overlap2}.

Theorem \ref{bords} will be proved in Section \ref{proofBords}.

\subsection{Consequences}

We display a direct consequence of Theorem \ref{bords}.

\begin{corollary}\label{power}
Any border of a conjugate of a Christoffel word is a power of a 
conjugate of a Christoffel word.
\end{corollary} 

As noted by one of the referees, this result may also be obtained as follows: if $u$ is a border of the Christoffel word $w$, 
then $uu$ is a factor of $ww$; since $ww$ is a Sturmian word, it is Sturmian, and thus $uu$ too. 
Hence all conjugates of $u$ are Sturmian. Among them is the power $\ell^k$ of some Lyndon word $\ell$, which is therefore Sturmian; 
hence $\ell$ is a Christoffel word by a theorem of Berstel and de Luca (see \cite{Re} Corollary 13.4.3).
%

One may be more precise. For this we need a notation, since we deal with different sequences $a_1,\ldots,a_m$. We write
$$
H_N(a_1,\ldots,a_m)=C^N(M_m),
$$
where $M_m=V_m(0,\ldots,0)$ is as before the word corresponding to the sequence $a_1,\ldots,a_m$ and to the Ostrowski representation of 0. Note that here $N$ may be in $
\mathbb Z$; but the word $H_N(a_1,\ldots,a_m)$ depends only on $N$ modulo $q_m$, where $q_m$ is the length of this word (recall that $q_j=K_j(a_1,\ldots, a_j)$).

\begin{corollary}\label{H} Let $a_1\ldots,a_m$ be a sequence of positive integers. If $N=q_{m-1}-1$ or $N=q_m-1$, then $H_N(a_1,\ldots,a_m)$ is a Christoffel word and has no border.

Now, let $0\leq N\leq q_m{-1}$, $N\neq q_{m-1}-1,q_m-1$, and denote by $B_N$ the longest border of $H_N(a_1,\ldots,a_m)$. 

\noindent (a) Suppose that $a_m\geq 2$. 

If $0\leq N< q_{m-1}-1$, then
$B_N=  H_N(a_1,\ldots,a_{m-1},a_{m}-1)$.

If $q_{m-1}\leq N< q_m{-1}$, then $B_N=  H_N(a_1,\ldots,a_{m-1})^t$, where $t =\min\{\lfloor \frac{N}{q_{m-1}}\rfloor,1+\lfloor \frac{q_m-2-N}{ q_{m-1}}\rfloor\}$.

\noindent (b) Suppose that $a_m=1$. 

If $0\leq N< q_{m-1}-1$, then
$B_N=  H_N(a_1,\ldots,a_{m-2})^t$, where $t=1+\min\{\lfloor \frac{N}{q_{m-2}} \rfloor,\lfloor \frac{q_{m-1}-2-N}{q_{m-2}} \rfloor \}$.

If $q_{m-1}\leq N< q_m{-1}$, then $B_N=  H_N(a_1,\ldots,a_{m-1})$.
\end{corollary}

Note that we recover a step function (the number $t$ in the statement) as it appears in Lapointe's article, see for example \cite[Figure 3]{La}.

Recall that $E$ denotes the involution which permutes $a$ and $b$.

\begin{lemma}\label{a1=1} 
Suppose that $a_1=1$. 
Let $V_i$, $i=0, \ldots , m$, be as usual and $V'_i$, $i=0, \ldots , m-1$, 
the sequence of words associated with the sequence of positive numbers $a_2+1, a_3, \ldots, a_{m}$.     
Then for any legal Ostrowski representation $N=\sum_{1\leq i\leq m} d_iq_{i-1}$ 
(so that $d_1=0$), one has
$$
V_m(d_1, d_2, \ldots,d_m) = E \bigl( V'_{m-1}(d_2,\ldots, d_{m}) \bigr).
$$
\end{lemma}

\begin{proof} Denote the new sequence by $a'_1=a_2+1,a'_2=a_3,\ldots,a'_{m-1}=a_m$; then the associated sequence of $b_i$'s is $b'_1=a'_1-1=a_2,b'_2=a'_2,\ldots,b'_{m-1}=a'_{m-1}$.
Recall that 
$V_{-1}=V'_{-1}=b,V_0 = V'_0 = a$. 
Observe that $V_1 = b$, $V_2 = b^{a_2 - d_2} a b^{d_2}$, while $V'_1 = a^{a_2 - d_2} b a^{d_2}$.
Thus, $V_1 = E (V'_0)$ and $V_2 = E (V'_1)$. 
An immediate induction based on (\ref{eq}) proves the lemma. 
\end{proof}

\begin{lemma}\label{am=1} Suppose that $a_m=1$. 
Let $V_i$, $i=1, \ldots , m$, be as usual and $V'_i$, $i=1, \ldots , m-1$,
the sequence of words associated with the sequence of positive numbers $a_1,\ldots,a_{m-2},a_{m-1}+1$. Then for any legal Ostrowski representation $N=\sum_{1\leq i\leq m} d_iq_{i-1}$ (so that $d_m=0$ or $1$), one has
$$
V_m(d_1,\ldots,d_{m-1},0)=V'_{m-1}(d_1,\ldots,d_{m-2},d_{m-1}+1), $$
$$V_m(d_1,\ldots,d_{m-1},1)=V'_{m-1}(d_1,\ldots, d_{m-2},d_{m-1}).
$$
\end{lemma}

\begin{proof} 
Let $q_i$ be as usual and write $q'_i$ for the corresponding numbers with respect to the sequence $a_1,\ldots,a_{m-2},a_{m-1}+1$.
By stability, we have $V_i=V'_i$ for $i=1,\ldots,m-2$. Next, we see that $V_{m-1}=V_{m-2}^{b_{m-1}-d_{m-1}}V_{m-3}V_{m-2}^{d_{m-1}}$. 
Thus, 
$$
V_{m}
(d_1,\ldots,d_{m-1},0)=V_{m-1}V_{m-2}  =V_{m-2}^{b_{m-1}-d_{m-1}}V_{m-3}V_{m-2}^{d_{m-1}+1} \\
 =
{V'}_{m-2}^{b_{m-1}-d_{m-1}}
V'_{m-3}
{V'}_{m-2}^{d_{m-1}+1} 
$$
$$
 =V'_{m-1}(d_1,\ldots,d_{m-2},d_{m-1}+1). 
 $$
Moreover, 
$$
V_{m}
(d_1,\ldots,d_{m-1},1)=V_{m-2}V_{m-1} =V_{m-2}^{b_{m-1}-d_{m-1}+1}V_{m-3}V_{m-2}^{d_{m-1}} 
$$
$$ =
{V'}_{m-2}^{b_{m-1}-d_{m-1}+1}
V'_{m-3}
{V'}_{m-2}^{d_{m-1}} 
$$
$$
 =V'_{m-1}(d_1,\ldots,d_{m-2},d_{m-1}),
$$
and the proof is complete.
\end{proof}

\begin{proof}[of Corollary \ref{H}] 
We assume that $m\geq 2$, since the case $m=1$ is easy to handle directly. We exclude the case $m=2, b_1 = 0$, which, by Lemma \ref{a1=1}, reduces to the case $m=1$.

1. By Theorem \ref{Vm}, $H_N(a_1,\ldots,a_m)$ is equal to the word $V_m(d_1,\ldots,d_m)$, where $N$ has the greedy representation $N=\sum_{1\leq i\leq m}d_iq_{i-1}$.
By Corollary \ref{Cword} and Lemma \ref{altlazyineq} (i), the word $H_i(a_1,\ldots,a_m)$ is a Christoffel word if and only if $N$ is equal to $q_{m-1}-1$ or to $q_m-1$.

2. We study now the number $t$ in Part (a) of the statement. Since $N=d_mq_{m-1}+\sum_{1\leq i\leq m-1}d_iq_{i-1}$, we get $\lfloor \frac{N}{q_{m-1}}\rfloor=d_m$ by Lemma \ref{greedylazyineq} 
(i). 

Next, let $j=\sum_{1\leq i\leq m-1}d_iq_{i-1}$. Then
$q_m-2-N=b_mq_{m-1}+q_{m-2}-2-d_mq_{m-1}-j=(b_m-d_m)q_{m-1}+q_{m-2}-2-j$. We have therefore $p:=1+\lfloor \frac{q_m-2-N}{ q_{m-1}}\rfloor=1+b_m-d_m+\lfloor \frac{q_{m-2}-2-j}{ q_{m-1}}\rfloor$. 

We show that the numerator in the latter fraction is always in the interval $[-q_{m-1},q_{m-1})$, so that the integer part of this fraction is either $-$1 or 0, and we give the condition when it is $0$.
We have indeed $-q_{m-1}\leq q_{m-2}-2-j$, since 
$j\leq q_{m-1}-1$ by Lemma \ref{greedylazyineq} 
(i), so that $j+2\leq q_{m-1}+1\leq q_{m-1}+q_{m-2}$. Moreover, $q_{m-2}-2-j< q_{m-1}$, since the 
sequence $q_i$ is increasing. Also, if $d_{m-1}>0$, then $j\geq q_{m-2}$, hence $q_{m-2}-2-j<0$; and if $d_{m-1}=0$, then $j=\sum_{1\leq i\leq m-2}d_iq_{i-1}$, and 
$q_{m-2}-2-j<0$ if and only if $j\geq q_{m-2}-1$, which, by Lemmas \ref{altlazyineq} (i) and 
\ref{greedylazyineq} (i), is equivalent to the fact 
that the sequence $d_1,\ldots,d_{m-2},0$ is alternating.

It follows that 
$p=b_m-d_m$, except if $d_{m-1}=0$ and if the sequence $d_1,\ldots,d_{m-2},0$ is not alternating, in which case 
$p=b_m-d_m+1$. Note that in all cases, $p=b_m-d_m$ or $p=b_m-d_m+1$.

3. We deduce from the previous part of the proof that $t=\min\{d_m,b_m-d_m\}=\ell$ (defined in Theorem \ref{bords}), except if $d_{m-1}=0$, if $b_m-d_m<d_m$, and if the sequence $d_1,\ldots,d_{m-2},0$ is not alternating, in which case $t=\ell+1$. This follows since if $b_m-d_m\geq d_m$, then $t=\min\{d_m,p\}=d_m=\min\{d_m,b_m-d_m\}=\ell$, because $p=b_m-d_m$ or $p=b_m-d_m+1$.

4. We assume that $a_m\geq 2$. Suppose that $0\leq N<q_{m-1}-1$. Then $d_m=0$ by Lemma \ref{greedylazyineq} (i). By Theorem \ref{bords} (v), we have $B_N=V_{m-1}^{b_m-1}V_{m-2}=H_N(a_1,\ldots,a_{m-1},a_m-1)$.

Suppose now that $q_{m-1}\leq N <q_m$. Then $d_m>0$ by Lemma \ref{greedylazyineq} (i). We are therefore in case (i), (ii), (iii) or (iv) of Theorem \ref{bords}. Note that $V_{m-1}=H_J(a_1,\ldots,a_{m-1})$, where $J=\sum_{1\leq i\leq m-1}d_iq_{i-1}$, so that $V_{m-1}=H_N(a_1,\ldots,a_{m-1})$, because $N$ is congruent to $J$ modulo $q_{m-1}$, the length of $V_{m-1}$. Thus we have to show that $B_N=V_{m-1}^t$.

In case (i), $B_N=V_{m-1}^t$; indeed, $d_m=b_m$ implies $d_{m-1}=0$ by greedyness, and since the sequence $d_1,\ldots,d_m$ is not alternating, $t=1$ by Part 3.

In case (ii), $B_N=V_{m-1}^\ell$, and $\ell=t$ by Part 3.

In case (iii), we have $d_{m-1}=0$ and $B_N=V_{m-1}^\ell$, except in the following case: $b_m-d_m<d_m$, the sequence $d_1,\ldots,d_{m-1}$ is 
not alternating, and then $B_N=V_{m-1}^{\ell+1}$. Thus $B_N=V_{n-1}^t$ by Part 3. 

In case (iv), we have $B_N=V_{m-1}^\ell$ and $\ell=t$ by Part 3 since $d_{m-1}\neq 0$.

5.
We assume that $a_m=1$. Define $a'_i=a_i$ if $i=1,\ldots,m-2$ and $a'_{m-1}=a_{m-1}+1$. We denote 
by $q'_i$ and $M'_i$ the corresponding 
words and numbers. We have $q'_i=q_i$ for $i=1,\ldots,m-2$ and by (\ref{n-1}),  $q'_{m-1}=q_{m-1}+q_{m-2}
=q_m$. 

We have $H_N(a_1,\ldots,a_m)=C^N(M_m)=C^N(V_m(0,\ldots,0))=C^N(V'_{m-1}(0,\ldots,0,1))$ (by Lemma \ref{am=1}) $=C^N(C^{q_{m-2}}(M'_{m-1}))$ (by 
Theorem \ref{Vm} and because $q'_{m-2}=q_{m-2}$) $=C^{N+q_{m-2}}(M'_{m-1})=
H_{N+q_{m-2}}(a'_1,\ldots,a'_{m-1})$. Thus $H_N(a_1,\ldots,a_m)=H_{N'}(a'_1,\ldots,a'_{m-1})$ where $N'=N+q_{m-2}$.

Suppose that $0\leq N<q_{m-1}-1$. Then $q'_{m-2}=q_{m-2}\leq N'<q_{m-1}+q_{m-2}-1=q_m-1=q'_{m-1}-1$. It follows from Case (a) that $B_N=H_{N'}
(a'_1,\ldots,a'_{m-2})^t$, with $t=\min\{\lfloor \frac{N'}{q'_{m-2}}\rfloor,1+\lfloor \frac{q'_{m-1}-2-N'}{ q'_{m-2}}\rfloor \}$.
Note that we have
$H_{N'}(a'_1,\ldots,a'_{m-2})=C^{N'}(H_0(a'_1,\ldots,a'_{m-2}))=C^{N+q_{m-2}}
(H_0(a_1,\ldots,a_{m-2}))$ 
$=C^N(H_0(a_1,\ldots,a_{m-2}))$
(since the word is of length $q_{m-2}$) $=H_N(a_1,\ldots,a_{m-2})$; moreover, 
$t=\min\{\lfloor \frac{N+q_{m-2}}{q_{m-2}}\rfloor,1+\lfloor \frac{q_{m-1}+q_{m-2}-2-N-q_{m-2}}{ q_{m-2}}\rfloor\}$, which settles this case.

Suppose now that $q_{m-1}\leq N<q_{m}-1$. Then we have $H_N(a_1,\ldots,a_m)=H_{N'}(a'_1,\ldots,a'_{m-1})$ $=H_{N''}(a'_1,\ldots,a'_{m-1})$, where $N''=N'-q_m=N+q_{m-2}-q_m=N-q_{m-1}$, since the words have length $q_m$. Now $0\leq N''<q_{m-2}-1$.
Hence by the first part, $B_N= H_{N''}(a'_1,\ldots,a'_{m-1}-1)=C^{N''}(H_0(a_1,\ldots,a_{m-1}))=C^N(H_0(a_1,\ldots,a_{m-1})) =H_N(a_0,\ldots,a_{m-1})$, 
since the word has length $q_{m-1}$.
\end{proof}

\subsection{Proof of Theorem \ref{bords}}\label{proofBords}

We keep our notation and consider the word
$$
V_m = V_m (d_1, \ldots , d_m),
$$
where 
\begin{equation}\label{legal}
N=d_1 q_0 +  \ldots +  d_m q_{m-1}
\end{equation} is a legal representation. 
We keep in mind several facts: 

\smallskip

$(i)$ If the words $X, Y$ satisfy $XY = YX$, then $X$ and $Y$ are both integral powers of a same word. 

$(ii)$  
We know that $V_m$ is a primitive word, that is, there do not exist a word $Z$ and an integer 
$\ell \ge 2$ such that $V_m = Z^\ell$; indeed, by Lemma \ref{endoV}, this word is part of a basis of the free group, so cannot be a nontrivial power.
Moreover, if there are words $X, Y$ such that 
$X V_m Y = V_m V_m$, then $X$ or $Y$ is empty. 


$(iii)$ Any word of the form $V_m^u V_{m-1} V_m^v$ with $u, v$ nonnegative integers, is primitive. This follows for the same reason as in (ii).   

$(iv)$ If the length of $W$ satisfies $1 \le |W| < |V_m|$, then $W V_m$ is not 
a prefix of $V_m V_m$, nor is $V_m W$ a suffix of $V_m V_m$. 

$(v)$ The words $V_m V_{m-1}$ and $V_{m-1} V_m$ are different. 

\smallskip

For $k= 0, \ldots , m-1$, we
let $W_k$ (resp., $X_k$) denote the longest common prefix 
(resp., suffix) of $V_{k+1} V_k$ and $V_k V_{k+1}$. 

\begin{lemma}   \label{commonpref}
Put $Z_1 = ab$ and $Z_{-1} = ba$. 
Let $k= 0, \ldots , m-1$ be an integer. 
Then,
$$
V_{k+1} V_k = W_k Z_{(-1)^{k+1}} X_k, \quad V_k V_{k+1} = W_k Z_{(-1)^k} X_k. 
$$
The word $W_k$ factors as
$$
W_k = V_{k}^{b_{k+1} - d_{k+1}} V_{k-1}^{b_{k} - d_{k}} \cdots V_0^{b_1 - d_1}
$$
and its length $w_k$ is given by
$$
w_k = \sum_{j=1}^{k+1} (b_j - d_j) q_{j-1}. 
$$
The word $X_k$ factors as
$$
X_k = V_0^{d_1} \cdots V_{k-1}^{d_{k}} V_{k}^{d_{k+1}} 
$$
and its length $x_k$ is given by
$$
x_k = \sum_{j=1}^{k+1} d_j q_{j-1}. 
$$
\end{lemma}

Observe that $w_k + x_k = \sum_{j=1}^{k+1} b_j q_{j-1} = q_{k+1} + q_k - 2$ (Lemma \ref{greedylazyineq} (iii)). 

\begin{proof}
We prove the lemma by induction on $k$. Recall that $V_{-1} = b$, $V_0 = a$, 
and $V_1 = V_0^{b_1 - d_1} V_{-1} V_0^{d_1}=a^{b_1 - d_1} b a^{d_1}$. 
This implies that $V_0 V_1 = a^{b_1 -d_1 + 1} b a^{d_1}$ and $V_1V_0=a^{b_1 - d_1} b a^{d_1+1}$. Thus
$$
W_0 = a^{b_1 - d_1}=V_0^{b_1 - d_1}, \quad w_0 = b_1 - d_1, 
\quad X_0 = a^{d_1} = V_0^{d_1}, \quad x_0 = d_1,
$$
and
$$
V_1 V_0 = W_0 ba X_0 = W_0 Z_{-1} X_0, \quad V_0 V_1 = W_0 ab X_0 = W_0 Z_1 X_0.
$$
This shows that the lemma holds for $k=0$. 
Now let $k \ge 0$ be an integer with $k < m-1$.
Assume that $V_{k+1} V_k = W_k Z_{(-1)^{k+1}} X_k$ and $V_k V_{k+1} = W_k Z_{(-1)^k} X_k$.

Since $V_{k+2} = V_{k+1}^{b_{k+2} - d_{k+2}} V_k V_{k+1}^{d_{k+2}}$, we get from our inductive 
assumption that
$$
V_{k+2} V_{k+1} = V_{k+1}^{b_{k+2} - d_{k+2}} V_k  V_{k+1} V_{k+1}^{d_{k+2}} 
=V_{k+1}^{b_{k+2} - d_{k+2}} W_k Z_{(-1)^k} X_k V_{k+1}^{d_{k+2}}
$$
and 
$$
V_{k+1} V_{k+2} = V_{k+1}^{b_{k+2} - d_{k+2}} V_{k+1} V_k V_{k+1}^{d_{k+2}}
=V_{k+1}^{b_{k+2} - d_{k+2}} W_k Z_{(-1)^{k+1}} X_k V_{k+1}^{d_{k+2}}.
$$
This shows that
$$
W_{k+1} = V_{k+1}^{b_{k+2} - d_{k+2}} W_k, \quad X_{k+1} = X_k V_{k+1}^{d_{k+2}}.
$$
Furthermore, 
$$
V_{k+2} V_{k+1} = W_{k+1} Z_{(-1)^k} X_{k+1} = W_{k+1} Z_{(-1)^{k+2}} X_{k+1}
$$
and
$$
V_{k+1} V_{k+2} = W_{k+1} Z_{(-1)^{k+1}} X_{k+1}. 
$$
Since $q_j$ is the length of $V_j$, this proves the lemma. 
\end{proof}

\begin{lemma} \label{XWpal}
Let $k=0, \ldots , m-1$ be an integer. 
With the above notation, the word $X_k W_k$ can be expressed as
$$
X_k W_k = V_0^{d_1} \cdots V_{k-1}^{d_{k}} V_{k}^{b_{k+1}} V_{k-1}^{b_{k} - d_{k}} \cdots V_0^{b_1 - d_1}
$$
and is a palindrome. 
More precisely, it is the central word of the conjugation class of $V_kV_{k+1}$.
\end{lemma}

\begin{proof} 
The expression of $X_k W_k$ is an immediate consequence of Lemma \ref{commonpref}. 
Recall Pirillo's theorem: if the words $aub, bua$ are conjugate, 
then $u$ is a central word (\cite{P}, \cite[Theorem 15.2.5]{Re}).
By Lemma \ref{commonpref}, the words $aX_kW_kb$ and $bX_kW_ka$ are conjugate. This proves the lemma.
%
%
\end{proof}


Lemma \ref{XWpal} extends Lemma \ref{pal}, which 
corresponds to the case $d_1 = \ldots = d_k = 0$. 

We display a consequence of Lemma \ref{commonpref}.

\begin{corollary}\label{prefixVV}
If $V_{k+1}$ is a prefix of $V_k V_{k+1}$, then $d_{k+1} = 0$. 
If $V_{k+1}$ is a suffix of $V_{k+1} V_k$, then $d_{k+1} = b_{k+1}$. 
\end{corollary}

\begin{proof}
If $V_{k+1}$ is a prefix of $V_k V_{k+1}$, then the common prefix $W_k$ of $V_kV_{k+1}$ and $V_{k+1}V_k$ is of length
$w_k \ge q_{k+1}$; since $w_k+x_k=q_{k+1}+q_k-2$, we obtain $x_k \le q_k - 2$. 
By Lemma \ref{commonpref} this implies that $d_{k+1} = 0$. Similarly, if $V_{k+1}$ is a suffix of $V_{k+1} V_k$, then 
$w_k \le q_k - 2$ and $b_{k+1} - d_{k+1} = 0$. 
\end{proof}

%
%
%
%
%
Recall that alternating sequences have been defined in Section \ref{Onumer}.

\begin{corollary}    \label{pref}
Suppose that the representation (\ref{legal}) is greedy.
The word $V_{m-1} (d_1, \ldots , d_{m-1})$ is a prefix of 
$V_m (d_1, \ldots , d_m)$ if and only if $V_m (d_1, \ldots , d_m)$ is not the
Christoffel word $V_m (  \ldots , 0, b_{m-2}, 0, b_m)$. 
\end{corollary}

\begin{proof}
Observe that $V_{m-1}$ is a prefix of $V_m$ if and only if the common prefix $W_{m-1}$ 
of $V_m V_{m-1}$ and $V_{m-1} V_m$ has length at least $q_{m-1}$. In view of 
Lemma \ref{commonpref}, this common prefix has length 
$$
w_{m-1}=\sum_{j=1}^{m} (b_j - d_j) q_{j-1}. 
$$
The lemma then follows from Lemma \ref{Ost-C} and Corollary \ref{Cword}.
\end{proof}

\begin{corollary}\label{pref1} Suppose that the representation (\ref{legal}) is greedy.
The word $V_{m-1}$ is not a  prefix of the word $V_{m-2} V_{m-1}$ if and only if $d_{m-1}=0$ and the sequence $d_1,\ldots,d_{m-1}$ is alternating.
\end{corollary}

\begin{proof}
We apply Corollary \ref{pref} to the sequence $a_1,\ldots, a_{m-1},1$ and the words $V_{m-1}(d_1, \ldots , d_{m-1})$ and
$V'_m=V_m (d_1, \ldots , d_{m-1},1)$, so that $V'_m=V_{m-2}V_{m-1}$: thus the word
$V_{m-1}$ is a prefix of $V_{m-2} V_{m-1}$ if and only if $V'_m$ is not a Christoffel word; but, by Corollary \ref{Cword}, $V'_m$ is a Christoffel word if and only if the sequence $d_{1}, \ldots, d_{m-1},1$ is alternating; this means that $d_{m-1}=0$ and the sequence $d_1,\ldots,d_{m-1}$ is alternating.
\end{proof}

Let us state several result on borders. 

\begin{lemma}\label{YX} Let $i\geq 2$, $Y$ be
a primitive word and $X$ a prefix of $Y$. Then the longest border $B$ of $Y^iX$ is $Y^{i-1}X$. 
\end{lemma}

Recall that an {\it internal factor} of a word means a factor that is not a prefix nor a suffix.

\begin{proof} The word $Y^{i-1}X$ is a border of $Y^iX$. Suppose that $B$ is longer. It begins by $Y^{i-1}X$, hence by $Y$ since $i\geq 2$: $B=YB'$. Moreover, $Y^iX=UB$, where the length of $U$ satisfies $0<|U|<|Y|$. Thus $Y^iX=UYB'$ and we see that $Y$ is an internal factor of $YY$, a contradiction.
\end{proof}

Observe that a border of a border of a word $W$ is a border of $W$: the borders of $W$ are 
totally ordered by the relation ``being a border".

\begin{lemma}   \label{bordmax}
Let $W$ be a finite word and $V$ be its longest border. 

(i) The borders of $W$ are precisely $V$ and its borders.

(ii) If $V = U^\ell Z$ with $U$ primitive, $\ell\geq 1$, and $Z$ a proper, possibly empty, prefix of $U$, then 
the borders of $W$ are $V=U^\ell Z, U^{\ell - 1} Z, \ldots , U Z$ and the borders of $U Z$. 
\end{lemma}

\begin{proof}
(i) Let $X$ be a border of $W$ with $X \not= V$. Then $X$ is a prefix and a suffix of $W$, hence, being shorter than $V$, also a prefix and a suffix of 
$V$. Consequently, $X$
is a border of $V$. The converse follows from the observation before the lemma.

(ii) Follows from (i) and Lemma \ref{YX}.
\end{proof}

\noindent {\bf Remark.}
Observe that if $U U$ is a border of $W$ and $V$ a border of $U$, 
then $UV$ is not necessarily a border of $W$. A counterexample is given by $abaababbabaaba$, with $U = aba$
and $V=a$. 

In the following lemmas, we consider the legal representation (\ref{legal}) and put $V_m = V_m (d_1, \ldots , d_m)$, $V_{m-1} = V_{m-1} (d_1, \ldots , d_{m-1})$, 
and $V_{m-2} = V_{m-2} (d_1, \ldots , d_{m-2})$. 

\begin{lemma} \label{overlap} The word $V_m$ is neither an internal factor of $V_m V_{m-1}$, nor of $V_{m-1} V_m$.
\end{lemma}

\begin{proof} We may assume that $m\geq 2$. Suppose that $V_m$ is an internal factor of $V_m V_{m-1}$.     
Then $V_m V_{m-1} = X V_m Y$, with $X$ and $Y$ nonempty.
Since $V_{m-1}$ is shorter than $V_m$, there exist a suffix
$W$ of $V_m$ and a prefix $Z$ of $V_{m-1}$ 
such that $V_m = X W = W Z$. Note that $|X|=|Z|$. Since $V_m$ is primitive, it is not equal to one of its conjugates, thus
the words $X$ and $Z$ are different; moreover, $X$ is a prefix of $V_m$ and $Z$ is a prefix of $V_{m-1}$; since they have the same positive length and are different, $V_{m-1}$ is not a prefix of $V_m$. 
Consequently, we have $V_m = V_{m-2} V_{m-1}^{b_m}$, $V_{m-2}V_{m-1}^{b_m+1}=V_mV_{m-1}=XV_mY=XV_{m-2} V_{m-1}^{b_m}Y$. Since $b_m\geq 1$ 
and $Y$ is shorter than $V_{m-1}$ (because $V_m V_{m-1} = X V_m Y$, hence $|X|+|Y|=|V_{m-1}|$ and $X$ nonempty), $V_{m-1} Y$ is a suffix of $V_{m-1} V_{m-1}$. Since $Y$ is nonempty, $V_{m-1}$ is an internal factor of $V_{m-1}V_{m-1}$. 
This contradicts the primitivity of $V_{m-1}$.

The proof for $V_{m-1}V_m$ is similar and we omit it. 
\end{proof}

\begin{lemma} \label{overlap2} Assume that $m\geq 1$.
Let $u, v$ be 
positive integers and set $V = V_{m}^u V_{m-1} V_{m}^v$. 
There is no other occurrence of $V_{m}$ in $V$, except possibly, one starting by $V_{m-1}$ (case L)
and one ending by $V_{m-1}$ (case R). Case L occurs if and only $V_m$ is a prefix of $V_{m-1}V_m$, and then $d_m=0$.
Case R occurs if and only if $V_m$ is a suffix of $V_mV_{m-1}$, and then $d_m=b_m$.
\end{lemma}

\begin{proof}
\noindent A) Consider an occurrence of $V_m$ in $V$. By the primitivity of $V_m$ and 
Lemma \ref{overlap}, suppose by contradiction that there exist nonempty words $X, Y$ such that 
$V_{m} = X V_{m-1} Y$, where $V_{m-1}$ is the factor appearing in the indicated factorization of $V$, $X$ is a suffix of $V_m$ and $Y$ a prefix of $V_m$. 

1. Assume first that $1 \le d_{m} \le b_{m} - 1$. Thus by (\ref{eqm}), the word 
$V_{m-1}$ is a prefix and a suffix of $V_m$. We show that
either $X$ is an integer power of $V_{m-1}$, or $Y$ is an integer power of $V_{m-1}$. 
Indeed, if $|X| < |V_{m-1}|$, then $X$ is a 
nontrivial proper suffix of $V_{m-1}$, $V_{m-1}=UX$, where $U$ is nonempty, and $V_m$ begins with $X V_{m-1}$; but $V_{m}$ also 
begins with $V_{m-1}$, $V_m=V_{m-1}W$, hence $UV_{m-1}W=UV_m=UXV_{m-1}Y=V_{m-1}V_{m-1}Y$; since $U$ is nonempty and shorter than $V_{m-1}$, we see that $V_{m-1}$ is a proper factor of $V_{m-1}^2$, and we have a contradiction with the primitivity of $V_{m-1}$. 

Consequently, $|X|\geq V_{m-1}$, and therefore $X=X'V_{m-1}$ (since $X$ and $V_{m-1}$ are both suffixes of $V_m$). A symmetric 
argument shows that $Y=V_{m-1}Y'$. Thus, $V_{m} = X' V_{m-1}^3 Y'$. Since $V_m = V_{m-1}^{b_m - d_m} V_{m-2} V_{m-1}^{d_m}$, $|V_{m-2}|\leq |V_{m-1}|$, and $V_{m-2}\neq V_{m-1}$, we see that $V_{m-1}$ is an internal factor of $V_{m-1}^2$, a contradiction 
with the primitivity of $V_{m-1}$. Hence, $X'$ or $Y'$ is an integer power of $V_{m-1}$. 

Assume that $X = V_{m-1}^z$, for some positive integer $z$, the other case being similar. We have two cases, depending on the relative values of $z$ and $b_m-d_m$. In
both cases, we claim that $V_{m-2}V_{m-1}$ is a prefix of $V_{m-1}^{2}$, a contradiction 
with the primitivity of $V_{m-1}$, since $V_{m-2}$ is not longer than $V_{m-1}$ and $V_{m-1}\neq V_{m-2}$. For the claim, we have indeed $V_m=XV_{m-1}Y=V_{m-1}^{z+2}Y'$ and $V_m=V_{m-1}^{b_m-d_m}V_{m-2}
V_{m-1}^{d_m}$.
If $b_m-d_m\leq z$, then $z=b_m-d_m+h$, $h\geq 0$, thus $V_{m-1}^{h+2}Y'=V_{m-2}V_{m-1}^{d_m}$, which proves the claim in this case, since $d_m\geq 1$. If $b_m-d_m> z$, then $b_m-
d_m=z+h+1$, $h\geq 0$, and $V_{m-1}^2Y'=V_{m-1}^{h+1}V_{m-2}V_{m-1}^{d_m}$, thus $Y=V_{m-1}Y'=V_{m-1}^hV_{m-2}V_{m-1}^{d_m}$; now $Y$ is a prefix of $V_m$, 
$V_m=YW$, hence $V_{m-1}^hV_{m-2}V_{m-1}^{d_m}W=V_m=V_{m-1}^{z+h+1}V_{m-2}V_{m-1}^{d_m}$, thus $V_{m-2}V_{m-1}^{d_m}W=V_{m-1}^{z+1}V_{m-2}V_{m-1}
^{d_m}$, which proves the claim, since $z,d_m\geq 1$.

2. Assume now that $d_m = 0$, hence $V_m = V_{m-1}^{b_m} V_{m-2}$. Since $V_m=XV_{m-1}Y$, we see that: either $Y$ is shorter than $V_{m-2}$ and then 
$V_{m-1}$ is an internal factor of $V_{m-1}V_{m-2}$, contradicting Lemma \ref{overlap}; or the length of $Y$ is larger than that of $V_{m-2}$, and noncongruent to it 
modulo $|V_{m-1}|$, and then $V_{m-1}$ is an internal factor of $V_{m-1}^2$, contradicting the primitivity of $V_{m-1}$; or the length of $Y$ is congruent to $|
V_{m-2}|$ modulo $|V_{m-1}|$, and then $X$ is an integral power of $V_{m-1}$.

Precisely, there are integers $r, s$ such that $r+s+1=b_{m}$,  $X = V_{m-1}^r$ and $Y = V_{m-1}^s V_{m-2}$. 
If $r \ge 2$, then $V_{m-1} V_{m-1}$ and $V_{m-1} V_{m-2}$ 
are suffixes of $V_m$, a contradiction with the 
primitivity of $V_{m-1}$. Thus, we have $r=1$ and $V_{m-2}$ 
is a prefix of $V_{m-1}$ (since $Y$, of length at most equal to $(s+1)|V_{m-1}|$, is a prefix of $V_m=V_{m-1}^{r+s+1}V_{m-2}$, hence of $V_{m-1}^{r+1+s}$). 
Observe that $X = V_{m-1}$ and $V_{m-1} V_{m-2}$ are suffixes 
of $V_m$. Since $V_{m-2}$ 
is a prefix of $V_{m-1}$, we get $V_{m-1} V_{m-2} = V_{m-2} V_{m-1}$, a contradiction.

3. The case $d_m = b_m$ is similar to the case $d_m = 0$ and we omit it.

\noindent B) Suppose now that there is an occurrence of $V_m$ starting at $V_{m-1}$. This means that $V_{m}$ is a prefix of $V_{m-1}V_m$. Then $d_m=0$ by Corollary 
\ref{prefixVV}.
%

Suppose now that there is an occurrence of $V_m$ ending at $V_{m-1}$; this is equivalent to the fact that $V_{m}$ is a suffix of $V_mV_{m-1}$. Then, similarly, we must have $d_m=b_m$.
\end{proof}

{\begin{lemma}\label{YXY} Let $i,j$ be positive integers, and $X,Y$ be nonempty words such that $X$ is shorter than $Y$, $Y$ is primitive, 
and $XY\neq YX$. Suppose further that in the word $W=Y^iXY^j$ 
there are at most
$i+j+2$ occurrences of the factor $Y$, namely the $i+j$ ones coming from the indicated factorization of $W$, and at most two others, 
beginning or ending by the $X$ indicated in the factorization (we denote these two cases respectively by L and R). 
Let $\ell=\min\{i,j \}$. Then the longest border $B$ of $W$ is $Y^{\ell +1}$ if either $i<j$ and case L occurs, or $i>j$ and case R occurs. In 
all other cases, $B=Y^\ell$.
\end{lemma}

Note that the cases L and R match with those of Lemma \ref{overlap2}.

\begin{proof} 1. Suppose by contradiction that $B=Y^iXY^r$, with $0\leq r\leq j$. Then $r<j$ since $B\neq W$. Moreover
$Y^iXY^{j-r}Y^r=W=UB=UY^{i-1} YXY^r$, hence $Y^iXY^{j-r}=UY^{i-1}YX$; if $j-r=1$, then $XY=YX$, a contradiction; thus $j-r \geq 2$ and, since $X$ is shorter 
than $Y$, we see that $Y$ is an internal factor of $YY$,
which contradicts the primitivity of $Y$.

We deduce that $B\neq Y^iXY^r$, when $0\leq r\leq j$, and by symmetry, $B\neq Y^rXY^j$, when $0\leq r\leq i$. 

2. We show that cases L and R cannot occur simultaneously. Indeed, if they occur together then, since $X$ is nonempty and shorter than $Y$, 
the factor $Y$ beginning at $X$ is an internal factor of $YY$, product of the factor $Y$ ending at $X$ and of the first factor $Y$ of $Y^j$; 
this contradicts the primitivity of $Y$. 

3. By symmetry, we may assume that 
$i\leq j$. Then $\ell=i$. 
Clearly, $Y^i$ is a border. 

Suppose that $B$ is longer; then $B$ extends $Y^i$ to the left, hence $B$ ends by $Y$, since $i\geq 1$;
moreover, we may extend the prefix $Y^i$ 
of $W$ to 
the longer prefix $B$, and since $B$ ends by $Y$, by 1. this $Y$ is the factor $Y$ of $W$ starting at $X$, and we are in case L; thus, since $B$ ends by $Y$, by 1. and by the hypothesis on the locations of the factors $Y$ in $W$, 
$B=Y^{i+1}$. But $B$ is also a right factor of $W$, hence we must have $j>i$, otherwise there is a factor $Y$ ending at $X$, which is excluded 
by 2. 
\end{proof}

\begin{lemma}\label{XY} Let $j\geq 1$, $Y$ be a primitive word, and $X$ a nonempty word, shorter that $Y$, such that $Y$ is a prefix of $XY^j$, and that $Y$ is not an internal factor of $XY$. Then the longest border $B$ of $XY^j$ is $Y$.
\end{lemma}

\begin{proof} We may write $B=YUZ$, 
where the length of $U$ is a multiple of that of $Y$, and $|Z|<|Y|$;
assume by contradiction that $Z$ is nonempty.

Then, since $B$ is a suffix of $XY^j$, we have $XY^j=W{\widetilde Y}UZ$; then we see that that either $\widetilde Y$ is an internal factor of $XY$ (a contradiction with the hypothesis), or $\widetilde Y$ is an internal factor of $YY$ (which contradicts the fact that $Y$ is primitive). Thus $Z$ must be empty.

It follows that $B=YU$, hence $B=Y^h$, since $B$ is a suffix of $XY^j$. If we have $h\geq 2$, then $j\geq 2$, and since $B$ is a prefix of $XY^j$, and $X$ is 
nonempty and shorter that $Y$, we see 
that $Y$ is an internal factor of $YY$, a contradiction again. Thus $h=1$.
\end{proof}

\begin{lemma}\label{dm0} If $m\geq 2$, $d_m=0$ and $V_{m-1}$ is a suffix of $V_{m}$, then $d_{m-1}=b_{m-1}$.
\end{lemma}

\begin{proof} We have by (\ref{eq}) $$V_m=V_{m-1}^{b_m}V_{m-2}.$$ Suppose that $m=2$. Then $V_{-1}=b,V_0=a,V_1=a^{b_1-d_1}ba^{d_1}$, 
$V_2=V_{1}^{b_2}a$, so that $V_2$ ends with $a^{b_1-d_1}ba^{d_1}a$; thus $V_1$ is not suffix of $V_2$.

Therefore $m\geq 3$. Note that, since $V_{m-1}$ and $V_{m-2}$ are both suffixes of the same word $V_m$, the word $V_{m-2}$ is a suffix of $V_{m-1}$. 
Let $V_{m-2}^h$ be suffix of $V_{m-1}$, with $h$ maximal; then $h\geq 1$ and $V_{m-2}^{h+1}$ is a suffix of $V_m$, since $b_m\geq 1$ because $m\geq 2$. We show that 
$|V_{m-2}^{h+1}|>|V_{m-1}|$. Indeed, otherwise $V_{m-2}^{h+1}$ is not longer than $V_{m-1}$, and since both $V_{m-2}^{h+1}$ and $V_{m-1}$ are suffixes of the same word 
$V_m$, the word $V_{m-2}^{h+1}$ is a suffix of $V_{m-1}$, contradicting the maximality of $h$. We thus deduce that $V_{m-1}=UV_{m-2}^h$, where $U$ is shorter than $V_{m-2}$. 

Recall that $V_{m-1}=V_{m-2}^{b_{m-1}-d_{m-1}}V_{m-3}V_{m-2}^{d_{m-1}}$. 
If $b_{m-1}-d_{m-1} \ge 2$ then, since $U V_{m-2}$ is a prefix of $V_{m-1}$, it is a prefix of $V_{m-2} V_{m-2}$, a contradiction 
with the primitivity of $V_{m-2}$ ($U$ is nonempty, otherwise either $V_{m-1}$ is not primitive, or $V_{m-1}=V_{m-2}$, a contradiction in both cases). 
Therefore we have $b_{m-1}-d_{m-1} = 1$, hence
$V_{m-1}=V_{m-2} V_{m-3}V_{m-2}^{b_{m-1}- 1}$. Then,
$$
V_m = V_{m-1}^{b_m} V_{m-2} = (V_{m-2}V_{m-3}V_{m-2}^{b_{m-1}-1})^{b_m}V_{m-2} = V_{m-2} (V_{m-3}V_{m-2}^{b_{m-1}})^{b_m},
$$
and, as $V_{m-1}$ is a suffix of $V_{m}$, we get, by comparing suffixes of the same length, the equality 
$V_{m-1} = V_{m-3}V_{m-2}^{b_{m-1}}$. Since also $V_{m-1}=V_{m-2} V_{m-3}V_{m-2}^{b_{m-1}- 1}$, we deduce that $V_{m-3}V_{m-2}=V_{m-2}V_{m-3}$, a contradiction.
Thus $b_{m-1}-d_{m-1} =0$.

\end{proof}

\begin{lemma}\label{infer} Suppose that the representation (\ref{legal}) is greedy. If 
$d_m=b_m$, then $V_m$ is not a suffix of $V_mV_{m-1}$.
\end{lemma} 

\begin{proof} We have $d_{m-1}=0$. Suppose that the lemma is false. We show first that $V_{m}$ is a 
Christoffel word. It is enough to prove that the $d_i$ are alternatively $b_i$ and 0 (Corollary \ref{Cword}). 
Since $V_{m}=V_{m-2} V_{m-1}^{b_{m}}$ 
is a suffix of $V_mV_{m-1}=V_{m-2} V_{m-1}^{b_{m} + 1}$, the word $V_{m-2}$ is a suffix of $V_{m-1}$. Note that for $m=1$, this cannot be true,
and neither for $m=2$, since $V_1=a^{b_1}b$, $V_0=a$;
hence we must have $m\geq 3$.
Since $d_{m-1} = 0$, we deduce from Lemma \ref{dm0}, applied to $m-1$, that $d_{m-2}=b_{m-2}$.
 
Moreover, since $V_{m-1}=V_{m-2}^{b_{m-1}}V_{m-3}$ and $V_{m-2}$ is a suffix of $V_{m-1}$, $V_{m-2}$ is a suffix of $V_{m-2}V_{m-3}$.

We thus obtain that $V_{m-2}$ is a suffix of $V_{m-2}V_{m-3}$ and that $d_{m-2}=b_{m-2}$. Continuing like this, we infer that $V_{m}$ is the 
Christoffel word $V_m (\ldots , 0, b_m)$.

To conclude, note that $V_{m-2}$ is a prefix and a suffix of $V_m$, contradicting the fact that a Christoffel word has no border.
\end{proof}

Now we are armed to prove Theorem \ref{bords}.

\begin{proof}[of Theorem \ref{bords}] 

$(i)$ If $d_m = b_m$, then $d_{m-1} = 0$ and we have
$$
V_m = V_{m-2} V_{m-1}^{b_m}, \quad V_{m-1} = V_{m-2}^{b_{m-1}} V_{m-3}. 
$$
Observe that $V_{m-1}=V_{m-2}^{b_{m-1}} V_{m-3}$ is a prefix of $V_{m} = V_{m-2}V_{m-2}^{b_{m-1}} V_{m-3} \cdots$ 
if and only if $V_{m-3}$ 
is a prefix of $V_{m-2}$, thus, by Corollary \ref{pref}, if and only if $V_{m-2}$ is not the
Christoffel word $V_{m-2}(\ldots, 0,b_{m-2})$. But $V_{m-2}$ cannot be equal to the latter word, since $V_m$ is by assumption not a Christoffel word, and $d_m=b_m,d_{m-1}=0$.
Thus $V_{m-1}$ is a prefix of $V_m$, and $B=V_{m-1}$ by Lemmas \ref{overlap} (applied to $m-1$) and \ref{XY}.

\smallskip 

$(ii)$ Suppose that $1 \le d_m \le b_m - 1$ and $1 \le d_{m-1} \le b_{m-1} - 1$.
There are no other occurrences of $V_{m-1}$ in $V_m$ than those seen in 
the factorization $V_m = V_{m-1}^{b_m - d_m} V_{m-2} V_{m-1}^{d_m}$;
indeed, this follows from Lemma \ref{overlap2} (applied to $m-1$), and our assumption on $d_{m-1}$.
Consequently, by Lemma \ref{YXY}, $B=V_{m-1}^\ell$. 
\smallskip

$(iii)$ 
If $1 \le d_m \le b_m - 1$ and $d_{m-1} = 0$, then $V_{m-1} = V_{m-2}^{b_{m-1}} V_{m-3}$. 
By Lemma \ref{overlap2} (applied to $m-1$), and the hypothesis $d_{m-1}=0$, any occurrence of $V_{m-1}$ in $V_m$ 
can be read on the factorisation $V_m = V_{m-1}^{b_m - d_m} V_{m-2} V_{m-1}^{d_m}$, or it begins 
by $V_{m-2}$. It follows from Lemma \ref{YXY} that $B=V_{m-1}^{\ell +1}$ if $V_{m-1}$ is a prefix of $V_{m-2}V_{m-1}$ and $b_m-d_m<d_m$, and otherwise $B=V_{m-1}^\ell$.

By Corollary \ref{pref1}, since $d_{m-1}=0$, $V_{m-1}$ is a prefix of $V_{m-2} V_{m-1}$ if and only if the sequence $d_1,\ldots,d_{m-1}$ is not alternating.

\smallskip 

$(iv)$ 
Suppose that $1 \le d_m \le b_m - 1$ and $d_{m-1} = b_{m-1}$. Then $V_{m-1} = V_{m-3} V_{m-2}^{b_{m-1}}$. 

We claim that there are no other occurrences of $V_{m-1}$ in $V_m$ than those given by the factorization
$V_m = V_{m-1}^{b_m - d_m} V_{m-2} V_{m-1}^{d_m}$. The claim is proved below. It follows from the claim and from Lemma \ref{YXY} that $B=V_{m-1}^\ell$.

By Lemma \ref{overlap2}, to prove the claim, it is enough to show that $V_{m-1}$ is not a prefix of $V_{m-2}V_{m-1}$, nor a suffix of $V_{m-1} V_{m-2}$.

This is immediate if $m=2$ and $b_1 \ge 1$, since we then get $V_0 = a$ and 
$V_1 = b a^{b_1}$. Thus, we assume $m \ge 3$.

By contradiction, suppose first that $V_{m-1}$ is a prefix of $V_{m-2}V_{m-1}$.
Since $V_{m-2}$ is a suffix of $V_{m-1}$, 
$V_{m-1}V_{m-2}$ is equal to $V_{m-2} V_{m-1}$, a contradiction. 

Supppose now that $V_{m-1}$ is a suffix of $V_{m-1} V_{m-2}$. This contradicts Lemma \ref{infer}, applied to $m-1$, since $d_{m-2}=0$ by greedyness.

\smallskip

$(v)$ 
If $d_m = 0$, then we have
$$
V_m = V_{m-1}^{b_m} V_{m-2},
$$
and, by Corollary \ref{pref}, either $V_{m-1}$ is the Christoffel word $V_{m-1}(\ldots,b_{m-3},0,b_{m-1})$, or $V_{m-2}$ is a prefix of $V_{m-1}$. The former case is excluded, since $V_m$ would be a Christoffel word. In the latter case,
$V_{m-1}^{b_m - 1} V_{m-2}$ is a border of $V_m$, 
and since $b_m \ge 2$, by Lemma \ref{YX}, $B=V_{m-1}^{b_m-1} V_{m-2}$.

$(vi)$ We suppose from now on that $d_m=0$ and $b_m = 1$. Then 
$$
V_m = V_{m-1} V_{m-2} = V_{m-2}^{b_{m-1} - d_{m-1}} V_{m-3} V_{m-2}^{d_{m-1} + 1}
$$
and there are several cases to distinguish. 

If $m=2$ and $b_1 \ge 1$, then $V_2 = a^{b_1 - d_1} b a^{b_1 + 1}$ and 
$B = a^h = V_{m-2}^h$. Assume that $m \ge 3$.

If $b_{m-1} - d_{m-1} \ge 1$ and 
$1 \le d_{m-2} \le b_{m-2} - 1$, 
then it 
follows from Lemma \ref{overlap2} (applied to $m-2$) and the hypothesis on $d_{m-2}$, that there are no further occurrences 
of $V_{m-2}$ in 
$V_m$. Thus $B=V_{m-2}^h$ by Lemma \ref{YXY}.

If $b_{m-1} - d_{m-1} \ge 1$ and $d_{m-2} = 0$, then by Lemmas \ref{overlap2} and \ref{YXY}, $B=V_{m-2}^{h+1}$ 
if $b_{m-1}-d_{m-1}< d_{m-1}+1$ and $V_{m-2}$ is a prefix of $V_{m-3}V_{m-2}$, and $B=V_{m-2}^h$ otherwise. But, by Corollary \ref{pref1} with $m$ replaced by $m-1$, $V_{m-2}$ is a prefix of $V_{m-3}V_{m-2}$ if and only if 
the sequence $d_1,\ldots,d_{m-2}$ is not alternating.

If $b_{m-1} - d_{m-1} \ge 1$ and $d_{m-2} = b_{m-2}$, then by Lemma \ref{infer} with $m$ replaced by $m-2$, $V_{m-2}$ is not a suffix of $V_{m-2}V_{m-3}$. Thus by Lemmas 
\ref{overlap2} and \ref{YXY}, $B=V_{m-2}^h$.

$(vii)$ We have $m \ge 3$: indeed, for $m=2$, $V_2=b a^{b_1 + 1}$ is a Christoffel word, which was excluded.
Since $d_{m-1} = b_{m-1}$, then $d_{m-2} =0$ by the greedy condition, and $V_m = V_{m-3} V_{m-2}^{b_{m-1} + 1}$.
If $V_{m-2}$ is not a prefix of $V_{m-3}V_{m-2}$, then by Corollary \ref{pref1}, the sequence $d_1,\ldots,d_{m-2}$ is alternating; then, since $d_{m-2}=0, d_{m-1}=b_{m-1}, d_m=0$, the sequence $d_1,\ldots,d_m$ is alternating too, and $V_m$ is a Christoffel word, a contradiction. Thus  $V_{m-2}$ is a prefix of $V_{m-3}V_{m-2}$, and 
by Lemmas \ref{overlap} and \ref{XY}, we get that $B=V_{m-2}$. 
\end{proof}

\section{The Sturmian graph revisited}\label{revisited}

We turn now to the suffixes of the central palindrome $p$ corresponding to a given Christoffel class. For this, we define $L_m=\widetilde M_m$, the reversal of the word $M_m$, with the previous notations. 

\begin{corollary}\label{suffix} Each suffix of $p$ has a unique factorization
$$
L_0^{d_1}L_1^{d_2}\cdots L_{m-1}^{d_m}
$$
where $\sum_{1\leq i\leq m}d_iq_{i-1}$ is the lazy Ostrowski representation of its length. In particular
\begin{equation}\label{L}
p=L_0^{c_1}L_1^{c_2}\cdots L_{m-1}^{c_m},
\end{equation}
where the $c_i$ are defined at the end of Section \ref{construct}.
\end{corollary}

\begin{proof} Let $s$ be a suffix of $p$, of length $N=\sum_{1\leq i\leq m}d_iq_{i-1}$, its lazy Ostrowski representation. Then $\widetilde s$ is a prefix of $p$. By Frid's result
(Corollary \ref{frid}), we have $\widetilde s=M_{m-1}^{d_m}\cdots M_0^{d_1}$. Applying the reversal mapping, which is an anti-automorphism, we obtain $s=L_0^{d_1}
L_1^{d_2}\cdots L_{m-1}^{d_m}$. Uniqueness follows from the uniqueness of the lazy representation.

The last assertion follows from the equality $q_m-2=\sum_{1\leq i\leq m}c_iq_{i-1}$, see Lemma \ref{greedylazyineq} (iii). 
\end{proof}

The previous corollary has a graph-theoretic interpretation. 
We construct an edge-labelled directed graph $(V,E)$, that we shall call {\it compact graph} for short.
It will turn out to be a graph introduced in \cite{EMSV}, where it is called the 
{\it compact directed acyclic word graph of  $p$}. 

For the construction of this graph, it is convenient to view (\ref{L}) as a word 
over  
the letters $L_0,\ldots,L_{m-1}$; in particular we consider prefixes of this 
word, which are the elements of $V$; the latter set has therefore $c_1+\cdots+c_m+1$ elements. 
We denote by $1$ the vertex corresponding to the empty word.
For each vertex $UL_i, 0\leq i\leq m-1$, there is an edge labelled 
$L_i$ from $U$ to $UL_i$:
$$
U \xrightarrow[]{L_i} UL_i.
$$ 
Moreover, if $i<m-1$ and $k\geq 1$, then for each vertex of the form $UL_i^kL_{i+1}, k\geq 1$, there is an edge labelled $L_{i+1}$ from $U$ to 
$UL_i^kL_{i+1}$: 
$$
U \xrightarrow[]{L_{i+1}} UL_i^kL_{i+1}.
$$
The construction is illustrated in Figure \ref{stgraph}.

We call the vertex $1$ the {\it origin}.
The label of a path in this graph is as usual the product of the labels of the edges of this path.

\begin{figure}
\begin{tikzpicture}
           
           \draw [fill=pink] (-4.5,0) circle (1 mm);
           \draw (-3.5,0) circle (1 mm);
           \draw (-1,0) circle (1 mm);
           \draw [fill=pink] (0,0) circle (1 mm);         
           \draw (1,0) circle (1 mm);
           \draw (3.4,0) circle (1 mm);
           \draw  [fill=pink] (4.4,0) circle (1 mm);
           \draw   (5.4,0) circle (1 mm);

           \draw (-5.8,0) node {$\cdots$};
           \draw[->,>=stealth'] (-5.4,0) to (-4.6,0);
           \draw[->,>=stealth']  [thick,red] (-4.4,0) to (-3.6,0);
           \draw[->,>=stealth']  [thick,red] (-3.4,0) to (-2.6,0);
           \draw (-2.2,0) node {$\cdots$};
           \draw[->,>=stealth'] [thick,red] (-1.9,0) to (-1.1,0);
           \draw[->,>=stealth'] [thick,red] (-0.9,0) to (-0.1,0);
           \draw[->,>=stealth'][thick,blue] (0.1,0) to (0.9,0);
           \draw[->,>=stealth'] [thick, blue](1.1,0) to (1.9,0);
           \draw (2.2,0) node {$\cdots$};
           \draw[->,>=stealth'] [thick,blue](2.5,0) to (3.3,0);
           \draw[->,>=stealth'] [thick, blue](3.5,0) to (4.3,0);
           \draw[->,>=stealth'] [thick,green] (4.5,0) to (5.3,0);
           \draw (5.8,0) node {$\cdots$};
           
           \draw[->,>=stealth'][thick,blue] (-4.5,-0.11) to [bend right=20]  (1,-0.11);
           \draw[->,>=stealth'] [thick,blue](-3.5,-0.11) to [bend right=15]   (1,-0.11);
           \draw[->,>=stealth'] [thick,blue](-1,-0.11) to [bend right=10]   (1,-0.11);
           
           \draw[->,>=stealth'] [thick,green] (0,0.11) to [bend left=30]   (5.3,0.11);
           \draw[->,>=stealth'] [thick,green] (1,0.11) to [bend left=30]   (5.3,0.11);
           \draw[->,>=stealth'] [thick,green] (3.4,0.11) to [bend left=30]   (5.3,0.11);

           \draw (-5,0) node [above]{$L_{i-1}$};
           \draw (-4,0) node [above]{$L_{i}$};
           \draw (-3,0) node [above]{$L_{i}$};
           \draw (-1.5,0) node [above]{$L_{i}$};
           \draw (-0.5,0) node [above]{$L_{i}$};
           \draw (0.5,0) node [above]{$L_{i+1}$};
           \draw (1.5,0) node [above]{$L_{i+1}$};
           \draw (2.9,0) node [above]{$L_{i+1}$};
           \draw (3.9,0) node [above]{$L_{i+1}$};
           \draw (4.9,0) node [above]{$L_{i+2}$};
           
           \draw (-2,-1) node {$L_{i+1}$};
           \draw (3,1.2) node {$L_{i+2}$};
           
           \draw (0,-2) node {{\it blue arrows are all labelled} $L_{i+1}$};
           \draw (0,-2.5) node {{\it green arrows are all labelled} $L_{i+2}$};
           \draw (0,-3) node {{\it a pink node separates $L_{j-1}$-arrows from $L_j$-arrows}};
           \draw (0,-3.5) node {{\it for any $j=1,\ldots m$, there are $c_{j+1}$ horizontal arrows labelled} $L_{j}$};

\end{tikzpicture}
\caption{The compact graph $(V,E)$}
\label{stgraph}
\end{figure}

\begin{corollary}\label{compact} For each suffix $s$ of $p$, there is a unique path in the compact graph, starting from the origin, and with label $s$.
\end{corollary}

\begin{proof} We know that, for $i=0,\ldots,m-1$, the last letter of the word $M_i$ is alternatively $a$ and $b$ (Corollary \ref{standard}). Hence the first letter of $L_i$ is alternatively $a$ and $b$. By 
construction, each vertex has at most two outgoing edges, and then they are labelled $L_i$ and $L_{i+1}$. Thus the graph has the following {\it deterministic} property: for each 
vertex, and for any two edges starting from it, the labels of these edges begin by distinct letters. This property ensures that for each word, there is at most one path starting from 
the origin and having this word as label. This proves uniqueness in the statement.

Consider some path from the origin in the graph. By inspection of the graph in Figure \ref{stgraph}, its label $s$ is a product $L_0^{d_1}\cdots L_{i+1}^{d_{i+2}}\cdots L_{m-1}^{d_m}$, where for any $j=0,\ldots,m-1$, $d_{j+1}$ is the number of edges labelled $L_{j}$  in the path; hence $0\leq d_{j+1}\leq c_{j+1}\leq b_{j+1}$. Thus $N=\sum_{1\leq i\leq m}d_iq_{i-1}$ is a legal Ostrowski representation of the length $N$ of $s$. This representation is lazy: indeed, suppose that for some $i\geq 0$, $d_{i+2}=0$ (with $i+2\leq m$); this means that the path has no edge labelled $L_{i+1}$; looking at the figure (where these edges are blue), we see that either the path has no vertex at the right of the central pink vertex (and then for all $j\geq i+2$, $d_j=0$), or the path must pass through this vertex, which implies that the path passes through all $L_i$-edges (red in the figure), and therefore $d_{i+1}=c_{i+1}=b_{i+1}$ (the last equality holds
since $i+1<m$). Hence the representation is lazy, and by Corollary \ref{suffix}, $s$ is a suffix of $p$.

Let now $s$ be any suffix of $p$. Then by Corollary \ref{suffix}, $s$ is equal to $L_0^{d_1}L_1^{d_2}\cdots L_{m-1}^{d_m}$, where $\sum_{i=1}^m d_iq_{i-1}$ is the lazy Ostrowski representation of the length of $s$. Let $k$ be maximal such that $d_k\neq 0$. Then $s=L_0^{d_1}L_1^{d_2}\cdots L_{k-1}^{d_k}$. We claim that for each $j=1,\ldots,k$, there is a path labelled $L_0^{d_1}\cdots L_{j-1}^{d_j} $ from the origin until the vertex $L_0^{c_1}\cdots L_{j-2}^{c_{j-1}}L_{j-1}^{d_j}$. 
The claim is clear for $j=1$, since one has the edges $1\to L_0 \to L_0^2 \to \cdots \to L_0^{c_1}$, all labelled $L_0$ and since $d_1\leq b_1=c_1$. Admitting the claim for $j\leq 
k-1$, we prove it for $j+1\leq k$. If $d_{j+1}=0$, then $j+1<k$ and by laziness, $d_j=b_j=c_j$ (the last equality 
holds since $j<m$); then the path for $j+1$ is the same as that for $j$: 
there is a path labelled $L_0^{d_1}\cdots L_{j-1}^{d_j} =L_0^{d_1}\cdots L_{j-1}^{d_j} L_j^{d_{j+1}}$ from the origin until the vertex $L_0^{c_1}\cdots L_{j-2}^{c_{j-1}}L_{j-1}^{d_j}
=L_0^{c_1}\cdots L_{j-2}^{c_{j-1}}L_{j-1}^{c_j}L_j^{d_{j+1}}$. Suppose now that $d_{j+1}\neq 0$; in the graph we have the $c_{j+1}$ consecutive edges $L_0^{c_1}\cdots L_{j-2}^{c_{j-1}}L_{j-1}
^{d_j} \to L_0^{c_1}\cdots L_{j-2}^{c_{j-1}}L_{j-1}^{c_j}L_j \to L_0^{c_1}\cdots L_{j-2}^{c_{j-1}}L_{j-1}^{c_j}L_j^2 \to\cdots \to L_0^{c_1}\cdots L_{j-2}^{c_{j-1}}L_{j-1}^{c_j}L_j 
^{c_{j+1}}$, all labelled $L_{j}$; note that  $d_{j+1}\leq c_{j+1}$: indeed, the representation is legal, hence $d_{j+1}\leq b_{j+1}$ and $b_{j+1}=c_{j+1}$, except if $j+1=m$; but in this case, since $s$ is of length at most $q_m-2$, we have 
$d_m\leq b_m-1=c_m$ by Corollary \ref{lazyBig}; thus the claim follows for $j+1$ too.

Thus, for $j=k$, we obtain that there is a path starting from the origin and labelled $s$, in the graph.
\end{proof}

In the compact graph $(V,E)$, replace each label of an edge by its length. We obtain a graph whose edges 
are labelled by positive natural numbers. This time, the sum of the labels of the edges of a path is called the {\it label} 
of this path. Since the suffixes of $p$ have all distinct lengths, we obtain

\begin{corollary}\label{SturmianGraph} For each natural number $N=0,1,\ldots,q_m-2$ there is a unique path in this graph, starting from the origin, with label $N$.
\end{corollary}

The compact graph is the {\it Sturmian graph} of \cite{EMSV,EFGMS}. This will be verified now. 

We define the notion of {\it generalized automaton}: it is a directed graph, whose vertices are called {\it states}, whose edges are called {\it transitions} and are labelled by nonempty words, with a distinguished vertex called the 
{\it initial state}, and a distinguished subset of the vertices, called the {\it set of final states}; the automaton is called {\it deterministic} if for any two edges outgoing from a vertex, 
their labels have distinct first letters; the generalized automaton is called {\it homogeneous} if for each vertex, the incoming edges all have the same label. The language 
recognized by a generalized automaton is the set of words which are labels of some path from the initial state to some final state. 

The compact graph is a deterministic homogeneous generalized automaton. Its initial state is the empty word, 
and each state is final; it recognizes the set of suffixes of $p$, by Corollary \ref{compact}. 

We may turn this generalized automaton into an {\it automaton} $\mathcal A$ (that is, where all the labels of the edges are letters), as follows: using Figure \ref{stgraph}, note that there is a maximal 
horizontal path labelled $L_{i+1}$:
$$q_0   \overset{L_{i+1}}{\to}    q_1   \overset{L_{i+1}}{\to}  \cdots \overset{L_{i+1}}{\to} q_c,$$
where $c=c_{i+2}$ is the number of horizontal edges labelled $L_{i+1}$ (the blue edges) in the compact graph.
Replace this path by an horizontal path whose edges are labelled by the letters of $L_{i+1}^{c}$, adding enough new vertices and new edges:
$$q_0   \overset{x}{\to}    q'  \cdots \overset{y}{\to} q_c,$$
where $x$ is the first letter of $L_{i+1}$, and $y$ its last.
Now, let each curved blue edge in the figure point onto the vertex $q'$, and have new label $x$.
The initial state of $\mathcal A$ is unchanged, and similarly for the final states.

A moment's thought shows that this new automaton $\mathcal A$ is deterministic, homogeneous, and recognizes the same language as the compact graph, that is, the set of suffixes of 
$p$. This automaton has $|p|+1$ vertices (because  there is in $\mathcal A$ a path labelled $p$ containing all vertices); hence it is minimal, in the sense that it has the smallest 
number of vertices among all automata recognizing this language: indeed, such an automaton must have at least $|p|+1$ vertices. 

There is a simple algorithm to recover the compact graph from the minimal automaton $\mathcal A$ of the set of suffixes of $p$: one chooses some vertex $v$ which is not final, 
which has only one outgoing edge $v\overset{t}{\to}v'$; one considers all incoming edges, all labelled by the same letter $z$ (since the automaton is homogeneous); then one 
suppresses the vertex $v$ and one lets the incoming edges point towards $v'$, adding $t$ at the end of their label. Iterating this procedure, called {\it compaction}, one 
recovers $(V,E)$.

In the light of \cite{EMSV} (Theorem 19, and beginning of Section 19, where compaction is described\footnote{The notion of compaction of an automaton appears in \cite{BBHME}.}), this proves that the graph of Corollary \ref{SturmianGraph} is the Sturmian graph. It implies also Theorem 47 of \cite{EFGMS}: each path in the Sturmian graph, with label $N$, corresponds to the lazy Ostrowski representation of $N$.

We indicate now how to construct the compact graph using the {\it iterated palindromization} of Aldo de Luca \cite{dL} (see also \cite[Section 12.1]{Re}). Recall the definition of this operator, denoted $Pal$. One defines first the {\it right palindromic closure} of a word $w$, denoted $w^{(+)}$: it is the shortest palindrome having $w$ as 
prefix. Then the mapping $Pal$ from a free monoid into itself is defined recursively by $Pal(1)=1$ and $Pal(wx)=(Pal(w)x)^{(+)}$ for any word $w$ and any letter $x$. The 
theorem of de Luca is that $Pal$ is a bijection from $\{a,b\}^*$ onto the set of central words.

If $p=Pal(v)$, $v$ is called the {\it directive word} of the central word $p$.
It follows from the definition of $Pal$ that the palindromic prefixes of $Pal(v)$ are the words $Pal(u)$, 
where $u$ runs through the prefixes of $v$.

\begin{proposition}\label{directive} The central word $p$ has the directive word $v=a^{c_1}b^{c_2}a^{c_3}\cdots (a \, \mbox{or} \,\, b)^{c_m}$. The word $p=Pal(v)$ has $1+c_1+\cdots+c_m$ 
palindromic prefixes, which are the formal prefixes of (\ref{L}). In particular, $L_i=Pal(a^{c_1}\cdots (a \, \mbox{or} \,\, b)^{c_{i}})^{-1}Pal(a^{c_1}\cdots (a \, \mbox{or} \,\, 
b)^{c_{i}}(b \, \mbox{or} \,\, a))$.
\end{proposition}

\begin{proof} We know that the Slope of $M_m$, and in particular of the lower Christoffel word in the conjugation class of $M_m$, is $S=[0,a_1,\ldots,a_n]$ (Theorem \ref{Vm}).  
It follows from the analysis at the end of Section \ref{Cword0} that 
$s=[0,b_1,\ldots,b_m]$ if $b_1\geq 1$, and $s=[b_2,\ldots,b_m]$ if $b_1=0$. It follows from Theorem 14.2.3 in 
\cite{Re} (the result is from \cite[Section 14.2.3]{GKP})
that the path leading from the root to the node $s$ in the Stern-Brocot tree is coded by the word 
$v=a^{c_1}b^{c_2}\cdots (a \, \mbox{or} \,\, b)^{c_m}$ ($a$ means left, and $b$ means right). It follows then from the correspondence between the Stern-Brocot tree, the tree of 
Christoffel words, and the tree of central words (see Sections 12.1, 14.1 and 14.2 in \cite{Re}) that the path from the root to $p$ in the latter tree is coded by $v$, 
proving the first assertion.

The word $p=Pal(v)$ has $|v|+1$ palindromic prefixes. It follows from Lemma \ref{pal} that all the words $M_{i-1}^{c} M_{i-2}^{c_{i-1}}\cdots M_{1}^{c_2} M_0^{c_1}$, where $i=1,\ldots,m$, $0\leq c\leq c_i$, are palindromes. Hence their reversals $L_0^{c_1}L_1^{c_2}\ldots L_{i-2}^{c_{i-1}}L_{i-1}^{c}$ are palindromes too, and are suffixes of $p$ by Corollary \ref{suffix}, proving the second assertion.

The last assertion then follows.
\end{proof}

The proposition implies that the compact graph, and the Sturmian graph, are embedded in the tree of central words, and in the Stern-Brocot tree.

\begin{corollary}\label{embedded} Consider in the tree of central words (resp. the Stern-Brocot tree) the path form the root to $p$ (resp. to the slope $s$ of $M_m$). Direct the edges downwards and label each edge $u\to v$ (resp. $p/q\to p'/q'$) by $u^{-1}v$ (resp. by $p'+q'-p-q$). Add an edge from each vertex to the first vertex after the first turn below on the path; the label of a new edge depends only on its final vertex. This graph is the compact graph (resp. the Sturmian graph).
\end{corollary}

\begin{proof} Consider some node on the tree of central words; as in the proof above, we associate with it the word $v\in \{a,b\}^*$, which encodes the path from the root to this vertex;  then this vertex is $Pal(v)$ (see \cite[Section 12.1]{Re}). The construction of the compact graph then follows from the proposition. And from this the construction of the Sturmian graph also follows. 
\end{proof}

An example may be useful. Let $m=3,a_1=2,a_2=1,a_3=3$. Then $M_{-1}=b,M_{0}=a,M_1=M_0^{a_1-1}M_{-1}=ab$, $M_2=M_1^{a_2}M_0=aba$, 
$M_3=M_2^{a_3}M_1=abaabaabaab$. Since $M_3=pab$, we have $p=aba^2ba^2ba$. The palindromic prefixes of $p$ are $1$, $a$, $aba$, $abaaba$, $p$; 
the letter following each palindromic prefix is underlined: $p=\underline a \,\underline ba\underline a ba\underline aba$, and therefore $p=Pal(abaa)$. One has 
$Pal(1)=1,Pal(a)=a,Pal(ab)=aba, Pal(aba)=abaaba$. The tree interpretation is shown in Figure \ref{tree}: the words $1,a,aba,abaaba,p$ are the nodes on the 
path from the root to $p$ in the tree of central words. One recovers in two ways that $L_0=a,L_1=ba,L_2=aba$: using $L_i=\widetilde M_i$, or using the last 
assertion of Proposition \ref{directive}.

\begin{figure}
\begin{tikzpicture}
\draw[-]  (-5.2,-0.2) to (-5.9,-0.9);
\draw[-]  (-5.8,-1.2) to (-5.1,-1.9);
\draw[-]  (-5.2,-2.2) to (-5.9,-2.9);
\draw[-]  (-6.2,-3.2) to (-6.9,-3.9);

\draw (-7,-4) node {$abaabaaba$};
\draw (-6,-3) node {$abaaba$};
\draw (-5,-2) node {$aba$};
\draw (-6,-1) node {$a$};
\draw (-5,0) node {$\epsilon$};
\draw[->]  (-1.1,-0.1) to (-1.9,-0.9);
\draw[->,>=stealth']  (-1.8,-1.2) to (-1.1,-1.9);
\draw[->,>=stealth']  (-1.2,-2.2) to (-1.9,-2.9);
\draw[->,>=stealth']  (-2.2,-3.2) to (-2.9,-3.9);
\draw[->,>=stealth'] (-0.9,-0.1) to [bend left=60]  (-0.9,-1.9);
\draw[->,>=stealth'] (-2.1,-1.1) to [bend right=60]  (-2.2,-2.9);
\draw (-3,-4) node {$abaabaaba$};
\draw (-2,-3) node {$abaaba$};
\draw (-1,-2) node {$aba$};
\draw (-2,-1) node {$a$};
\draw (-1,0) node {$\epsilon$};
\draw (-1.6,-0.4) node {$a$};
\draw (-1.30,-1.3) node {$ba$};
\draw (-0.2,-1) node {$ba$};
\draw (-3,-1.8) node {$aba$};
\draw (-1.2,-2.5) node {$aba$};
\draw (-2.2,-3.5) node {$aba$};
\draw[->]  (2.9,-0.1) to (2.1,-0.9);
\draw[->,>=stealth']  (2.2,-1.2) to (2.9,-1.9);
\draw[->,>=stealth']  (2.8,-2.2) to (2.1,-2.9);
\draw[->,>=stealth']  (1.8,-3.2) to (1.1,-3.9);
\draw[->,>=stealth'] (3.1,-0.1) to [bend left=60]  (3.1,-1.9);
\draw[->,>=stealth'] (1.9,-1.1) to [bend right=60]  (1.8,-2.9);
\draw (1,-4) node {$\frac{4}{7}$};
\draw (2,-3) node {$\frac{3}{5}$};
\draw (3,-2) node {$\frac{2}{3}$};
\draw (2,-1) node {$\frac{1}{2}$};
\draw (3,0) node {$\frac{1}{1}$};
\draw (2.3,-0.4) node {$1$};
\draw (2.7,-1.3) node {$2$};
\draw (3.8,-1) node {$2$};
\draw (1.2,-1.9) node {$3$};
\draw (2.7,-2.6) node {$3$};
\draw (1.6,-3.7) node {$3$};
\end{tikzpicture}
\caption{A path in the tree of central words, a compact graph and a Sturmian graph}
\label{tree}
\end{figure}

\section{Appendix: a proof of existence and uniqueness of the greedy and lazy representations}\label{app}

We want to prove Proposition \ref{greedylazy}. We begin by two lemmas.

\begin{lemma}\label{altlazyineq} Let $k=0,\ldots,m$ and a legal Ostrowski representation
\begin{equation}\label{Ost2}
N=d_1q_0+d_2q_1+\cdots+d_kq_{k-1}.
\end{equation}

(i) If in (\ref{Ost2}) the sequence $d_i$ is alternating, with $k=0$ or $d_k\neq 0$, then $N=q_k-1$.

(ii) If in (\ref{Ost2}) one assumes that the representation is lazy and that $k=0$ or $d_k=b_k$, then $N\geq q_k-1$.
\end{lemma}

Note that we say that the representation (\ref{Ost2}) is legal (resp. greedy, resp. lazy) if the representation (\ref{Ost1}) of $N$ obtained by letting $d_i=0$ for 
$i=k+1,\ldots,m$ has this property.

\begin{proof} (i) The hypothesis implies $d_k=b_k$. For $k=0$ and $k=1$, the equality follows from $q_{0}=1$ and $b_1=a_1-1=q_1-1$. Suppose that $k\geq 1$, and that the equality is true for $k-1$ and $k$. 
Consider an alternating sequence $d_1,\ldots,d_{k+1}$ with $d_{k+1}\neq 0$; then $d_{k+1}=b_{k+1}$. We have $\sum_{i=1}^{k+1} d_iq_{i-1}=  b_{k+1}q_k+\sum_{i=1}^{k-1} d_iq_{i-1}$ (since $d_k=0$, because 
the sequence is alternating)  $=a_{k+1}q_k+q_{k-1}-1$  (by induction, since $d_{k-1}=b_{k-1}\neq 0$) $=q_{k+1}-1$.

(ii) This is clearly true for $k=0$ and $k=1$, since $q_{0}=1$ and $b_1=a_1-1=q_1-1$. Assume that $k\geq 1$ and that it is true for $k-1$ and $k$, and we prove it for $k+1$; thus we consider a 
sequence $d_1,\ldots,d_{k+1}$ with $d_{k+1}=b_{k+1}$. If $d_k\neq 0$, then $N=\sum_{i=1}^{k+1} d_iq_{i-1}\geq b_{k+1}q_k+q_{k-1}=a_{k+1}q_k+q_{k-1}= q_{k+1}\geq q_{k+1}-1$. If $d_k=0$ then, 
assuming that $k\geq 2$, we have 
$d_{k-1}=b_{k-1}$ since the representation is lazy
; thus by induction, $N\geq b_{k+1}q_k +q_{k-1}-1=a_{k+1}q_k+q_{k-1}-1=q_{k+1}-1$. The remaining case is $k=1, d_2=b_2=a_2, d_1=0$ and $N=d_2q_1=a_2a_1=q_2-1$.
\end{proof}

\begin{lemma}\label{greedylazyineq} Let $0\leq k\leq m$, $N\in \N$, and $N=\sum_{i=1}^{k} d_iq_{i-1}$ be a legal Ostrowski representation.

(i) If the representation is greedy, then 
$$
N \leq q_k-1;
$$
if moreover $k=0$ or $d_k\neq 0$, then
$$
q_{k-1}-1<N.
$$

(ii) If $k\geq 1$ and the representation is lazy, then
$$
N\leq q_k+q_{k-1}-2;
$$
if moreover, $d_k\neq 0$, then
$$
q_{k-1}+q_{k-2}-2 <N.
$$

(iii) One has
$$\sum_{i=1}^{k}b_iq_{i-1}=q_k+q_{k-1}-2.
$$
\end{lemma}

\begin{proof} 
(i) 
We prove the first inequality by induction on $k$. 
For $k=0$, $N=0$ and it holds since $q_0=1$.
For $k=1$, it holds since $d_1\leq a_1-1$. 
Assume that $k\geq 1$, and that it is true for $1,\ldots,k$, and we prove it for $k+1$. If $d_{k+1}=a_{k+1}=b_{k+1}$ (since  $k+1\geq 
2$), then $d_k=0$ by (\ref{greedy}); then $d_1q_0+d_2q_1+\cdots+d_{k+1}q_{k}=d_1q_0+d_2q_1+\ldots+d_{k-1}q_{k-2}+a_{k+1}q_k\leq $ (by induction) 
$q_{k-1}-1+a_{k+1}q_k=q_{k+1}-1$; if on the other hand, $d_{k+1}\leq a_{k+1}-1$, then $d_1q_0+d_2q_1+\cdots+d_{k+1}q_{k}=d_1q_0+d_2q_1+\cdots+d_kq_{k-1}+d_{k+1}
q_k\leq $ (by induction) $
q_k-1+a_{k+1}q_k-q_k<-1+q_{k-1}+a_{k+1}q_k=q_{k+1}-1$.

The second inequality follows from $q_{-1}=0$, and from $d_k>0$ if $k\geq 1$.

(ii) If $k=1$, both inequalities are easy to verify. Suppose that they hold for $k\geq 1$, and consider the case $k+1$, $N=\sum_{i=1}^{k+1} d_iq_{i-1}$. By induction, $
\sum_{i=1}^{k} d_iq_{i-1}\leq q_k+q_{k-1}-2$, hence, since $d_{k+1}\leq a_{k+1}$, $N\leq q_k+q_{k-1}-2+a_{k+1}q_k=q_{k+1}+q_k-2$. 

Suppose now that $d_{k+1}\neq 0$. Then, if $d_{k}\geq 1$, then 
$N=d_{k+1}q_k+d_kq_{k-1}+\cdots \geq q_k+q_{k-1}> q_k+q_{k-1}-2$. Suppose now that $d_k=0$; if $k\geq 2$, we have $d_{k-1}=b_{k-1}$ by lazyness, hence by
Lemma \ref{altlazyineq} (ii) $\sum_{i=1}^{k-1} d_iq_{i-1}\geq q_{k-1}-1$; thus $N=d_{k+1}q_k+\sum_{i=1}^{k-1} d_iq_{i-1}\geq q_k+q_{k-1}-1>q_k+q_{k-1}-2$. The remaining case is $k=1$, $d_2\geq 1$, $d_1=0$; then $N=d_2q_1=d_2a_1\geq a_1> a_1-1=q_1+q_0-2 $.

(iii) is proved similarly by induction.
\end{proof}

\begin{corollary}\label{lazyBig} Let $k\geq 1$. For a lazy representation $N=\sum_{i=1}^{k} d_iq_{i-1}$, one has $d_k=b_k$ if and only if $N\geq q_k-1$.
\end{corollary}

\begin{proof} Suppose that $d_k=b_k$; then $N\geq q_k-1$ by Lemma \ref{altlazyineq} (ii). 

Suppose now that $d_k\neq b_k$; then $d_k\leq b_k-1$. Then $N+q_{k-1}$ has the lazy representation $(d_k+1)q_{k-1}+\sum_{i=1}^{k-1} d_iq_{i-1}$. Thus by Lemma \ref{greedylazyineq} (ii), $N+q_{k-1}\leq q_{k}+q_{k-1}-2$, hence $N\leq q_{k}-2$.
\end{proof}

\begin{proof}[of Proposition \ref{greedylazy}]  We observe that the sequence $q_k$, $k=-1,0,1,2,\ldots, m,$ is strictly increasing, except that one can have $q_0=q_1=1$, and this happens 
if and only if $a_1=1$.

(i) Let $0\leq k\leq m$. We prove the existence of a greedy representation $N=\sum_{i=1}^{k}d_iq_{i-1}$ for each $N$ satisfying $N \le q_{k} - 1$; 
by the previous observation, this will prove the existence 
of a greedy representation for each $N$ with $0\leq N\leq q_m-1$. For $k=0$, 
we have $N=0$ and existence is clear. For $k=1$, $N\leq q_1-1$; then $N\leq a_1-1$ and we have $N=d_1q_0$, $d_1=N\leq a_1-1$ and existence follows.

Suppose now that $k\geq 1$, and let $N$ satisfy $N \le q_{k+1} - 1$. 
Then $N\leq a_{k+1}q_k+q_{k-1}-1$. Since $q_{k-1}-1<q_k$, we have $N<(a_{k+1}+1)q_k$; thus,
performing the Euclidean division of $N$ by $q_k$, there are uniquely
determined $r, t$ with 
$N = t q_k + r$, $0 \le r \le q_k - 1$ and $t\leq a_{k+1}$. 

By induction on $k$, $r$ has a greedy representation $r=\sum_{i=1}^{k}d_iq_{i-1}$, and then $N$ has the representation obtained by adding that of $r$ and $tq_k$. If $t < a_{k+1}$, it is a greedy representation. If $t=a_{k+1}$, then
we have $a_{k+1}q_k+r=N\leq q_{k+1}-1=a_{k+1}q_k+q_{k-1}-1$, thus $r\leq  q_{k-1} - 1$, and $N = r + 0 \cdot q_{k-1} + a_{k+1} q_k$, and we conclude by induction on $k$ that $r$ has a greedy representation, hence $N$ too.

We prove now the uniqueness of the greedy representation. We may assume that $N\neq 0$. 
Assume that
we have two greedy representations for $N$, $N=\sum_{i=1}^{k}d_iq_{i-1}$, $N=\sum_{i=1}^{h}
e_iq_{i-1}$, written in such a way that $d_k\neq 0\neq e_h$. We have by Lemma \ref{greedylazyineq} (i): $N< q_k$, $N< q_h$, $N\geq q_{k-1}$, and $N\geq q_{h-1}$. This 
forces $k=h$, since the sequence $(q_i)$ is increasing. By Lemma \ref{greedylazyineq} (i), we have $r=\sum_{i=1}^{k-1} d_iq_{i-1}<q_{k-1}$; since $N=d_kq_{k-1}+r$, $d_{k}$ is the quotient of the Euclidean 
division of $N$ by $q_{k-1}$; similarly for $e_k$, so that
$d_{k}=e_k$, and the representations coincide by induction on $k$, since the greedy condition remains if one replaces the highest nonzero digit by 0.

(ii) We prove now the existence of the lazy representation. We observe that $(*)$ the sequence $q_{k} + q_{k-1} -2$ is strictly increasing for $k=0,\ldots,m$, with first value $-1$. 
We prove by induction on $k=1,\ldots,m$ that if 
$N\leq q_k+q_{k-1}-2$, then $N$ has a lazy representation of the form $N=\sum_{i=1}^{k} d_iq_{i-1}$. For $k=1$, the inequality is $N\leq 
a_1-1$, and we have indeed $N=d_1q_0$, with $d_1=N$, $0\leq d_1\leq a_1-1$. Assume now that $k\geq 1$, and that the property holds for $k$, and we prove it when $N$ satisfies $ N \le q_{k+1} + q_k - 2$. By induction, we may assume that $q_k+q_{k-1}  - 2 <N$. We have  $q_k +q_{k-1} - 1 \leq N$ and since $q_{k-1}-1\geq 0$ (because $k\geq 1$), 
there exists $j$, $1\leq j \leq a_{k+1}$ such that $jq_k \leq N$ and we take $j$ maximal. Then either $j=a_{k+1}$ and $N\leq q_k+q_{k+1}-2=(j+1)q_k+q_{k-1}-2$; or 
$j<a_{k+1}$ and then $j+1\leq a_{k+1}$ and by maximality, $N<(j+1)q_k\leq (j+1)q_k+q_{k-1}-1$ and we have $N\leq (j+1)q_k+q_{k-1}-2$, too.

Write $N = j q_k + N'$; then $0 \le N' \le q_{k} + q_{k-1} -2$. By induction, $N'$ has a lazy representation 
$
N' = d_1 q_0 + \cdots + d_k q_{k-1}.$
Then, $d_1 q_0 + \cdots + d_k q_{k-1} + j q_k$ is a lazy representation of $N$, except when $d_k = 0$ and $d_{k-1} \not= b_{k-1}$ (so that $k\geq 2$);
since $a_k=b_k$, $d_1 q_0 + \cdots + (d_{k-1} + 1) q_{k-2}  + b_k q_{k-1} + (j-1) q_k$ is then a lazy representation of $N$. 

We prove now uniqueness of the lazy representation. We may assume that $N\neq 0$. Suppose that $N$ has two lazy representations $N=\sum_{i=1}^{k}d_iq_{i-1}$, $N=\sum_{i=1}^{h}e_iq_{i-1}$, written in such a way that $d_k\neq 0\neq e_h$. We have by Lemma \ref{greedylazyineq} (ii): $q_{k-1}+q_{k-2}-2 < N\leq q_k+q_{k-1}-2$ and $q_{h-1}+q_{h-2}-2 < N\leq q_h+q_{h-1}-2$. This implies that $k=h$, by observation $(*)$. 

We claim that $b_k-d_k$ is the quotient of the Euclidean division of $N'=q_k+q_{k-1}-2-N$ by $q_{k-1}$. The same being true for $b_k-e_k$, we have $d_k=e_k$ and we conclude by induction that the representations coincide.  

For the claim, we may assume that $k\geq 2$; 
we have $N'=\sum_{i=1}^{k}(b_i-d_i)q_{i-1}$ by Lemma \ref{greedylazyineq} (iii). We have $N'=(b_k-d_k)q_{k-1}+r$, where $r=\sum_{i=1}^{k-1}(b_i-d_i)q_{i-1}$.
By lazyness, this is a greedy representation of $r$. Hence $r\leq q_{k-1}-1$ by Lemma \ref{greedylazyineq} (i); since $r\geq 0$, the claim follows.
\end{proof}

\medskip

\acknowledgments We thank the two anonymous referees for their comments.
This work was partially supported by NSERC, Canada.

\nocite{*}
\bibliographystyle{abbrvnat}
\bibliography{OstrowskiDef}
\label{sec:biblio}

\end{document}